\documentclass[11pt,reqno]{amsart}
\usepackage{fullpage}

\usepackage{graphicx}
\usepackage{hyperref}
\usepackage{amsmath,amsopn,amssymb,amsfonts,stmaryrd}
\usepackage{verbatim}
\usepackage{amsthm}
\usepackage{mathtools}
\usepackage{color}
\usepackage{simplewick}
\usepackage{enumitem}
\usepackage[framemethod=TikZ]{mdframed}
\usepackage{bbm}
\usepackage{mathrsfs}
\usepackage{booktabs}
\usepackage{caption}

\newcommand{\IR}{{\mathbb R}}%Reals
\newcommand{\IC}{{\mathbb C}}%Complex
\newcommand{\IZ}{{\mathbb Z}}%Integers
\newcommand{\IN}{{\mathbb N}}%Natural numbers
\newcommand{\IQ}{{\mathbb Q}}%Rationals
\newcommand{\IH}{{\mathbb H}}%quaternions
\newcommand{\CC}{\mathcal{C}}

\newcommand{\CE}{\mathcal{E}}

\renewcommand{\Im}{{\rm Im}}

\newcommand{\sgn}{\mbox{sgn}}

\newcommand{\SL}{\mathrm{SL}}

\newcommand{\zz}{\mathfrak{z}}

\theoremstyle{plain}
\newtheorem{thm}{Theorem}[section]
\newtheorem{cor}[thm]{Corollary}
\newtheorem{lem}[thm]{Lemma}
\newtheorem{prop}[thm]{Proposition}

\theoremstyle{definition}

\newtheorem*{rem}{Remark}

\numberwithin{equation}{section}

\newcommand{\pmat}[1]{\left( \smallmatrix #1 \endsmallmatrix \right)}

\renewcommand{\sgn}{\textnormal{sgn}}

\def\lp{\left(}
\def\rp{\right)}

\def\a{\alpha}
\def\b{\beta}
\def\d{\delta}

\def\k{\kappa}

\def\e{\varepsilon}

\def\n{\nu}

\def\t{\tau}

\def\del{  \partial}

\def\ddd{{\rm d}}

\newcommand{\im}{{\rm Im}}
\newcommand{\y}{{\rm \mathbbm{y}}}
\newcommand{\xx}{{\rm \mathbbm{x}}}

\renewcommand{\sgn}{{\rm sgn}}

\def\wt{\widetilde}
\def\wh{\widehat}
\def\bar{\overline}
\newcommand{\andd}{\quad \mbox{ and } \quad}

\def\slashchar#1{\setbox0=\hbox{$#1$}           % set a box for #1
	\dimen0=\wd0                                 % and get its size
	\setbox1=\hbox{/} \dimen1=\wd1               % get size of /
	\ifdim\dimen0>\dimen1                        % #1 is bigger
	\rlap{\hbox to \dimen0{\hfil/\hfil}}      % so center / in box
	#1                                        % and print #1
	\else                                        % / is bigger
	\rlap{\hbox to \dimen1{\hfil$#1$\hfil}}   % so center #1
	/                                         % and print /
	\fi}                                        %

\setlist[itemize]{noitemsep, topsep=0pt}

\newcounter{exercise}
\renewcommand{\theexercise}{\thesection.\arabic{exercise}}

\mdfdefinestyle{exercise}{%
	linecolor=blue!40,
	outerlinewidth=2pt,
	leftline=false,rightline=false,
	skipabove=\baselineskip,
	skipbelow=\baselineskip,
	frametitle=\mbox{},
}
\newmdenv[%
style=exercise,
settings={\global\refstepcounter{exercise}},
frametitlefont={\bfseries Exercise~\theexercise\quad},
]{exercise}
\newmdenv[%
style=exercise,
frametitlefont={\bfseries Exercise~\quad},
]{exercise*}

\newmdenv[%
backgroundcolor=gray!8,
linecolor=violet,
outerlinewidth=1pt,
roundcorner=3mm,
skipabove=\baselineskip,
skipbelow=\baselineskip,
]{boxes}

\makeatletter
\newcommand{\vast}{\bBigg@{2}}
\newcommand{\Vast}{\bBigg@{5}}
\makeatother

\renewcommand{\pmod}[1]{\  \,  \left(  \mathrm{mod} \,  #1 \right)}
\newcommand{\erf}{{\mathrm{erf}}}

\renewcommand{\b}[1]{\boldsymbol{#1}}



\title{A Framework for Modular Properties of False Theta Functions}

\author{Kathrin Bringmann}
\author{Caner Nazaroglu}
\address{University of Cologne, Department of Mathematics and Computer Science, Weyertal 86-90, 50931 Cologne, Germany}
\email{kbringma@math.uni-koeln.de}
\email{cnazarog@math.uni-koeln.de}

\thanks{The research of the first author is supported by the Alfried Krupp Prize for Young University Teachers of the Krupp foundation and the research leading to these results receives funding from the European Research Council under the European Union's Seventh Framework Programme (FP/2007-2013) / ERC Grant agreement n. 335220 - AQSER. The research of the second author is supported by the European Research Council under the European Union's Seventh Framework Programme (FP/2007-2013) / ERC Grant agreement n. 335220 - AQSER}

\begin{document}

\begin{abstract}
False theta functions closely resemble ordinary theta functions, however they do not have the modular transformation properties that theta functions have. In this paper, we find modular completions for false theta functions, which among other things gives an efficient way to compute their obstruction to modularity. This has potential applications for a variety of contexts where false and  partial theta series appear. To exemplify the utility of this derivation, we discuss the details of its use on two cases. First, we derive a convergent Rademacher-type exact formula for the number of unimodal sequences via the Circle Method and extend earlier work on their asymptotic properties. Secondly, we show how quantum modular properties of the limits of false theta functions can be rederived directly from the modular completion of false theta functions proposed in this paper.
\end{abstract}

\maketitle
\section{Introduction and Statement of results}
In this paper we embed false theta functions into a modular framework. \textit{False theta functions} are functions that are similar to theta functions, but they feature some different sign factors, which prevent them from being modular forms like theta functions.
Although there are some results  on how false theta functions behave under modular transformations \cite{CM}, an overarching derivation and explanation of their modular properties is so far missing. As false theta functions (and closely related  partial theta functions) have a rich history with applications in a variety of mathematical and  physical subfields (see e.g. \cite{Alladi,Andrews81,BCR2014, BM, Cheng:2018vpl, Chung:2018rea, CM, CMW, Fine, Garoufalidis2015, Hikami, LawrenceZagier, ZagierVass}),  understanding their modular properties is a worthy goal to pursue with potential implications for these applications. Our main result in the following discussion is to ``complete" false theta functions by viewing them as boundary values of modular objects that are functions of three complex variables with two variables in the upper half plane and one elliptic variable.

Despite their apparent non-modularity, false theta functions possess some modular transformation properties on the rationals. To describe this, recall that {\it quantum modular forms} \cite{ZaQ} are functions $f:\mathbb{Q} \to \mathbb{C}$ whose obstruction to modularity, $f(\t) - (c\t+d)^{-k} f (\frac{a\t+b}{c \t+d})$ for $\pmat{a & b \\ c & d} \in \mathrm{SL}_2 (\mathbb{Z})$ (or a vector-valued generalization with a suitable multiplier system), is ``nice''. In our setting ``nice'' means that the obstruction extends to a real-analytic function on the reals except for a finite set of points. An interesting and important source of examples is given via quantum invariants of knots and $3$-manifolds \cite{LawrenceZagier}. Rather generally, false theta functions yield quantum modular forms by taking vertical limits in the upper half-plane to rational numbers.
Modular properties, in turn, are explained by the fact that their asymptotic behavior is also related to mock theta functions.

Recall that mock theta functions  were initially introduced by Ramanujan in his last letter to Hardy  (see pages 127-131 of \cite{Ra}) with a  list of 17 examples and the words
\begin{quote}
	``I am extremely sorry for not writing you a single letter up to now \ldots I discovered very interesting functions recently which I call ``Mock'' theta functions. Unlike the ``False'' theta functions (studied partially by Prof. Rogers in his interesting paper) they enter into mathematics as beautifully as the ordinary theta functions. I am sending you with this letter some examples.''
\end{quote} 
Zwegers in his Ph.D. thesis \cite{Zw} viewed the mock theta functions as pieces of real-analytic functions, which transform like modular forms. To be more precise, the lack of modular invariance is repaired by adding non-holomorphic Eichler integrals of the shape ($\tau =\tau_1 + i \tau_2 \in \mathbb{H}$ with $\tau_1,\tau_2 \in \IR$)
\begin{equation}\label{nonE}
\int_{-\overline{\tau}}^{i\infty} \frac{\theta(w)}{\sqrt{-i(w+\tau)}} \ddd w
\end{equation}
to the mock theta functions, where  $\theta$ is a weight $\frac32$ unary theta function called the {\it shadow} of the mock theta function. 
The asymptotic relation between false theta functions and mock theta functions also goes through integrals of the form \eqref{nonE}. Incidentally, it is again this integral (or equivalently its Fourier expansion involving the incomplete gamma function) that we generalize in this paper to understand false theta functions.

Therefore, the challenge we tackle in this paper is to find modular completions for false theta functions to parallel the situation with mock theta functions in contrast to what one may predict from the above quote by Ramanujan. To illustrate our results on a simple yet representative example, consider the false theta function in two variables $\left( z\in\IC,q:=e^{2\pi i\tau},\zeta:= e^{2\pi i z}\right)$
\begin{equation*}
\psi(z;\tau) \coloneqq i \sum_{n\in\IZ} \sgn\left(n+\tfrac12\right)(-1)^n q^{\frac12\left(n+\frac12\right)^2} \zeta^{n+\frac12} .
\end{equation*}
Here we use the usual convention that $\sgn(x):=\pm 1$ for $\pm x > 0$ and $\sgn(0):=0$. Although the function $\psi$ is invariant under $T:=\left(\begin{smallmatrix}
1&1\\0&1
\end{smallmatrix}\right)$ it does not transform invariantly under $S:=\left(\begin{smallmatrix}
0&-1\\1&0
\end{smallmatrix}\right)$. To understand the problem, consider the formally similar Jacobi theta function
 \begin{equation*}
\vartheta(z;\tau) \coloneqq i \sum_{n\in\IZ} (-1)^n q^{\frac12\left(n+\frac12\right)^2} \zeta^{n+\frac12} .
\end{equation*}
This function is invariant under both the $T$ and the $S$ transformation and the invariance under the non-trivial of the two, namely $S: (z,\tau) \mapsto \lp \tfrac{z}{\t}, - \tfrac{1}{\t} \rp$, follows by using Poisson summation and self-duality of Gaussian functions under the Fourier transform. It is this self-duality property that is broken by the extra sign function, hence preventing invariance under $S$. 

Before giving the precise statement of our solution to repair modular invariance, it is worth describing a qualitative picture for it.
First, recall that a similar problem is also present for indefinite theta functions studied by Zwegers \cite{Zw}. The factors of sign functions that ensure convergence of indefinite theta functions are also what prevent them from transforming like  modular forms. The resolution in \cite{Zw} is to replace those sign functions with error functions (taking also $\sqrt{2\t_2}$ factors as input) giving rise to real-analytic modular completions. One way to understand how modular invariance is recovered is to note that the convolution of a sign function and a Gaussian is an error function. Therefore, the modular completion is basically a linear combination of Gaussian functions whose self-duality under the Fourier transform gives an object that behaves invariantly under the modular group. Note that this intuition for how modular invariance is recovered is quite general and suggests a similar remedy for false theta functions. The fact that sign functions are inserted with variables depending on positive definite directions in the lattice however create some problems, since simply changing the sign of $\t$ in the error functions (to compensate for the changing signature) leads to a divergent series. However, note that in the case of indefinite theta functions replacing $-\bar{\tau}$ in $\sqrt{2 \t_2} = \sqrt{-i (\t - \bar{\t})}$ with an independent complex variable $w \in \mathbb{H}$ again leads to a modular invariant object as long as one transforms $\t$ and $w$ simultaneously as $(\t, w) \mapsto (\tfrac{a \t+b}{c \t + d}, \tfrac{a w-b}{-c w + d})$ under $\pmat{a & b \\ c & d} \in \mathrm{SL}_2 (\mathbb{Z})$. The solution for false theta functions is then to use the same type of error function with the sign of $\t$ reversed but that of $w$ kept the same and then to require both variables to transform in the same way as $(\t, w) \mapsto (\tfrac{a \t+b}{c \t + d}, \tfrac{a w+b}{c w + d})$.

More concretely, for $\tau,w\in\IH$ and $z \in \IC$ define (see the discussion in Section \ref{sec:proofThm:Partials} for issues of convergence and holomorphicity) with $z_2 \coloneqq\im(z)$
\begin{equation}\label{definePh}
\wh{\psi}(z;\tau,w) \coloneqq i \sum_{n\in\IZ} \erf\left( -i \sqrt{\pi i(w-\tau)}\left(n+\tfrac12 + \tfrac{z_2}{\tau_2}\right)\right) (-1)^n  q^{\frac12\left(n+\frac12\right)^2}
\zeta^{n+\frac12},
\end{equation}
where $\erf (z) \coloneqq \frac{2}{\sqrt{\pi}} \int\limits_{0}^{z} e^{- t^2} dt$ denotes the \textit{error function} and where we define the square root using the principal branch. Note that 
\begin{equation}\label{limit}
\lim_{t\to\infty} \wh{\psi}(z;\tau, \tau+i t+\e) = \psi(z;\tau)
\end{equation}
if $-\frac12 < \frac{z_2}{\tau_2} < \frac12$ and $\e >0$ arbitrary. The following theorem gives the modular properties of $\wh{\psi}$, for a more precise version see Theorem \ref{Pht}.

\begin{thm}\label{partial1}
	The function $\widehat{\psi}$ transforms like a Jacobi form.
\end{thm}
Theorem \ref{partial1} is a special case of a more general result. To state this, let $L$ be a lattice of rank $N$ (which we can  identify with $\mathbb{Z}^N$) and $B:L \times L \to \IZ$ a positive-definite bilinear form with $Q(\boldsymbol{n}) \coloneqq \tfrac{1}{2} B(\boldsymbol{n}, \boldsymbol{n})$.
Note that throughout this paper we write vectors in bold letters and their components with subscripts. Let $\boldsymbol{\ell} \in L$ be a {\it characteristic vector}, i.e., $ Q(\boldsymbol{n}) + \frac12 B(\boldsymbol{\ell}, \boldsymbol{n}) \in \mathbb{Z}$ for all $\boldsymbol{n} \in L$.  Moreover $L^*$ denotes the dual lattice of $L$.
We let $\b{c}\in\IR^N$ be an arbitrary vector satisfying $2Q(\b{c})=1$. Define for $\b \mu \in L^*$
\begin{equation*}
\Psi_{Q, \boldsymbol{\mu}}( \boldsymbol{z}; \tau) =\Psi_{Q, \boldsymbol{\mu},\boldsymbol{\ell}, \boldsymbol{c}} ( \boldsymbol{z}; \tau)
 \coloneqq
\sum_{\boldsymbol{n} \in \boldsymbol{\mu} + \frac{\boldsymbol{\ell}}{2} + L} \sgn \lp B(\boldsymbol{c},\,  \boldsymbol{n}) \rp 
q^{Q(\boldsymbol{n})}   e^{2 \pi i B\left(\boldsymbol{n},\,\boldsymbol{z}+ \frac{\boldsymbol{\ell}}{2}\right)},
\end{equation*}
and its completion ($w\in\mathbb H$)
\begin{equation*}\begin{split}
\wh{\Psi}_{Q, \boldsymbol{\mu}}( \boldsymbol{z}; \tau,w)& =\wh{\Psi}_{Q,\boldsymbol{\mu},\boldsymbol{\ell},\boldsymbol{c}}( \boldsymbol{z}; \tau,w)\\
&\coloneqq \sum_{\boldsymbol{n} \in \boldsymbol{\mu} + \frac{\boldsymbol{\ell}}{2} +  L }
\erf \lp - i \sqrt{\pi  i (w- \tau)} 
B\left(\boldsymbol{c}, \boldsymbol{n} + \tfrac{\Im(\boldsymbol{z})}{\tau_2}\right) \rp 
 q^{Q(\boldsymbol{n})} e^{2 \pi i B\left(\boldsymbol{n},\,\boldsymbol{z}+ \frac{\boldsymbol{\ell}}{2}\right)}.
\end{split}\end{equation*}
We prove the following transformation law, which turns $\wh{\Psi}_{Q, \boldsymbol{\mu}}$ into a function transforming like a Jacobi form.
For this, let 
\begin{equation*}
\chi_{\tau,w} \coloneqq \sqrt{\tfrac{i (w - \tau)}{\tau w}}  \tfrac{\sqrt{\tau}\sqrt{w}}{\sqrt{i(w-\tau)}}.
\end{equation*}
\begin{thm}\label{ThPh}
	The function $\wh{\Psi}_{Q, \boldsymbol{\mu}}$ is a holomorphic function of $w$ (and of $\tau$ if $\boldsymbol{z} = 0$) away from the branch cut defined by $\sqrt{i (w - \tau)}$. It satisfies the following Jacobi transformations properties:
	\begin{enumerate}[leftmargin=*, label={\rm (\arabic*)}]
		\item For $\b m, \b r\in L$, we have
		\begin{align*}
		\wh{\Psi}_{Q, \boldsymbol{\mu}}(\boldsymbol{z} + \boldsymbol{m} \tau+\boldsymbol{r}; \tau,w) &= (-1)^{2Q(\boldsymbol{m}+\boldsymbol{r})}  q^{-Q(\boldsymbol{m})} 
		e^{-2\pi i B(\boldsymbol{m},\,\boldsymbol{z})} \wh{\Psi}_{Q, \boldsymbol{\mu}}(\boldsymbol{z}; \tau,w).
		\end{align*}
		\item We have
		\begin{align*}
		\wh{\Psi}_{Q, \boldsymbol{\mu}}( \boldsymbol{z}; \tau+1,w+1)  &= e^{2\pi i Q\left(\boldsymbol{\mu}+\frac{\boldsymbol{\ell}}{2}\right)}
		\wh{\Psi}_{Q, \boldsymbol{\mu}}(\boldsymbol{z};\tau,w),\\
		\wh{\Psi}_{Q, \boldsymbol{\mu}} \lp  \tfrac{\boldsymbol{z}}{\tau} ; -\tfrac{1}{\tau}, -\tfrac{1}{w} \rp  &= 
		\chi_{\tau,w}
		\frac{(-i \tau)^{\frac{N}{2}}}{\sqrt{| L^*/ L |}} 
		e^{2\pi i  \frac{Q(\boldsymbol{z})}{\tau}-\pi i Q(\boldsymbol{\ell})} 
		\sum_{\boldsymbol{\nu} \in L^* / L}  e^{- 2 \pi i B(\boldsymbol{\mu},\,\boldsymbol{\nu})}  
		\wh{\Psi}_{Q, \boldsymbol{\nu}} ( \boldsymbol{z}; \tau, w).
		\end{align*}
	\end{enumerate}
\end{thm}
We expect that completed false theta functions allow a theory parallel to that of mock theta functions. We will study this in future research and in this paper we give two applications, an exact formula for the number of unimodal sequences and quantum modular forms. We start with unimodal sequences.
A finite sequence of positive integers $\{a_j\}_{j=1}^s$ is called a {\it unimodal sequence of size $n$} if there exists $k\in \IN$ such that $a_1\leq a_2 \leq \ldots \leq a_k \geq a_{k+1} \geq \ldots \geq a_s$ and $a_1+\ldots+a_s=n$. Let $u(n)$ denote the number of unimodal sequences of size $n$. Then (see e.g. \cite{Au51})
\begin{equation*}
U(q) := \sum_{n\geq 0} u(n) q^n =\frac{1}{(q;q)_\infty^{\, 2}} \sum_{n\geq 1} (-1)^{n+1} q^{\frac{n(n+1)}{2}},
\end{equation*}
where $(q;q)_\infty:=\prod_{j=1}^\infty(1-q^j)$. Note that we may write
\begin{equation}\label{splitU}
U(q) = 
\frac{i}{2} q^{-\frac{1}{24}} \frac{\psi(\tau)}{\eta(\tau)^2}
+ \frac{q^{\frac{1}{12}}}{\eta (\tau)^2},
\end{equation}
where $\eta(\tau):=q^{\frac{1}{24}}\prod_{n\geq 1}(1-q^n)$ is {\it Dedekind's eta function} and $\psi(\tau):=\psi(0,\tau)$. The first few Fourier coefficients of each term are given by 
\begin{align*}
- \frac{i}{2}  \frac{\psi(\tau)}{\eta(\tau)^2} &= q^{\frac{1}{24}}
\lp 1 + q + 3 q^2 + 6 q^3 + 12 q^4 + 21 q^5 + 38 q^6 + 63 q^7 + 106 q^8
+170 q^9  + \ldots \rp,\\
\frac{1}{\eta(\tau)^2} &= q^{-\frac{1}{12}}
\lp 1 + 2q + 5 q^2 + 10 q^3 + 20 q^4 + 36 q^5 + 65 q^6 + 110 q^7 + 185 q^8
+300 q^9 + \ldots \rp.
\end{align*}
The leading asymptotics of $u(n)$ was determined by Auluck \cite{Au51} as
\begin{equation}\label{AW}
u(n) = \frac{1}{8\cdot 3^{\frac34} n^{\frac54}} e^{2\pi \sqrt{\frac{n}{3}}} \left(1+O\left(n^{-\frac12}\right)\right).
\end{equation}
This result was then generalized by Wright \cite{Wright}, who gave the asymptotic expansion to all orders of $n$ for the leading exponential term. In this paper, we prove an exact formula for $u(n)$, that is a convergent series expression which thereby includes all the subleading exponential contributions for large $n$. Our expression is analogous to the exact formula for the integer partition function $p(n)$ found by Rademacher \cite{Rademacher}, which is solely determined by the modular transformation properties of its generating function and the principal parts of the corresponding modular form at the cusps.

To state the exact formula for $u(n)$, we define for $n,r\in\IZ$ and $k\in\IN$ the {\it Kloosterman sums} 
\begin{equation*}\begin{split}
K_k (n) &\coloneqq 
i \sum_{\substack{0 \leq h < k \\ \mathrm{gcd} (h,k) = 1}}  
\nu_\eta (M_{h,k})^2 
\zeta_{12k}^{-(12n-1)h-h'},
\\
K_k(n,r) &\coloneqq  e^{\frac{3 \pi i}{4}} (-1)^r
\sum_{\substack{0 \leq h < k \\ \mathrm{gcd} (h,k) = 1}}  
\nu_\eta (M_{h,k})^{-1}  
\zeta_{24k}^{-(24n+1)h+\left(12 r^2 + 12 r + 1\right)h'},
\end{split}
\end{equation*}
where $h'$ is a solution of $hh' \equiv -1\pmod{k}$,
$
M_{h, k}:=\left(\begin{smallmatrix} h'&-\frac{hh'+1}{k}\\k&-h\end{smallmatrix}\right), \ \zeta_\ell :=e^{\frac{2\pi i}{\ell}}
$
for $\ell\in\IN$ and $\nu_\eta$ is the multiplier for $\eta$. In particular, for $M = \pmat{a & b \\ c & d}$ with $c>0$ it is given by (see Theorem 3.4 in \cite{Apostol})
\begin{equation*}
	\nu_\eta (M) := \exp \lp \pi i \lp \tfrac{a+d}{12c} - \tfrac{1}{4} + s(-d,c) \rp \rp,
	\quad \mbox{where }
	s(h,k) \coloneqq \sum_{r=1}^{k-1} \tfrac{r}{k} \lp  \tfrac{hr}{k} - \left\lfloor \tfrac{hr}{k} \right\rfloor - \tfrac12 \rp.
\end{equation*}
Finally $I_\kappa$ denotes the $I$-Bessel function of order $\kappa$. We then have the following expression for $u(n)$.
\begin{thm}\label{asTh}
	We have
	\begin{align*}
	u(n)=\frac{2 \pi}{12n-1} \sum_{k\geq 1}
	\frac{K_k(n) }{k} 
	I_2 \lp \tfrac{\pi}{3k} \sqrt{12n-1 }  \rp-
	\frac{\pi}{2^{\frac34} \sqrt{3}(24n+1)^{\frac34}} 
	\sum_{k\geq 1}  \  \  \sum_{r \pmod{2k}} \frac{K_k(n,r)}{k^2}\\
	\times
	\int_{-1}^{1}   \left(1 - x^2\right)^{\frac{3}{4}}
	\cot \lp \tfrac{\pi}{2k} \lp \tfrac{x}{\sqrt{6}}-r-\tfrac{1}{2} \rp \rp
	I_{\frac{3}{2}} \lp \tfrac{\pi}{3\sqrt{2}k} \sqrt{\left(1-x^2\right) \lp 24n + 1 \rp}  \rp 
	\ddd x .
	\end{align*}
\end{thm}
As a corollary we obtain \eqref{AW}.
\begin{cor}\label{asCor}
The asymptotics \eqref{AW} hold.
\end{cor}
\begin{rem}
Following the proof of Corollary \ref{asCor} one can determine further terms in the asymptotic expansion of $u(n)$.
\end{rem}
We next turn to explaining how quantum modularity follows from the construction of completed false theta functions. For simplicity we consider a special family studied by Milas and the first author \cite{BM}. Define, for $j \in \IZ$ and $N\in\IN_{>1}$,
\begin{align}\label{Fjp}
F_{j,N}(\tau)\coloneqq \sum_{\substack{n\in\IZ\\n\equiv j\pmod{2N}}} \sgn(n)  q^{\frac{n^2}{4N}}.
\end{align}
In fact, we can further restrict to $1 \leq j \leq N-1$, because $F_{j, N} = - F_{-j,N}$ and $F_{j+2N,N}= F_{j,N}$.
The following theorem describes the (false) modular behavior of the $F_{j,N}$. For this define the weight $\frac32$ unary theta functions
\begin{equation*}
f_{j,N}(\tau) \coloneqq  \frac{1}{2N}  \sum_{\substack{n\in\IZ\\n\equiv j\pmod{2N}}}nq^{\frac{n^2}{4N}} .
\end{equation*}
\begin{thm}\label{QuantumTheorem}
		For $M = \pmat{a & b \\ c & d} \in \mathrm{SL}_2 (\IZ)$, we have
		\begin{equation*}
		F_{j,N} \lp \t  \rp  
		-  \sgn (c \t_1 + d) 
		(c \tau +d)^{-\frac{1}{2}}  \sum_{r=1}^{N-1} \psi_{j,r}\left(M^{-1}\right)   F_{r,N} \lp \tfrac{a\t+b}{c \t +d } \rp
		=   -i
		\sqrt{2N} \int_{-\frac{d}{c}}^{i\infty} \frac{f_{j,N}(\mathfrak z)}{\sqrt{-i (\zz- \t)  }}\ddd\zz,
		\end{equation*}
		where the integration path avoids the branch cut defined by $\sqrt{-i (\zz- \t)  }$ and
		the multiplier $\psi_{j,r}$ is defined in \eqref{definemul}. 
	\end{thm}
As a corollary we obtain quantum modular properties of $F_{j,N}$.
\begin{cor}\label{quantum}
	The functions $F_{j,N}$ are vector-valued quantum modular forms with quantum set $\IQ$.
\end{cor}

The paper is organized as follows. In Section \ref{sec:proofThm:Partials} we prove Theorem \ref{ThPh}, which then in particular implies Theorem \ref{partial1}. Section \ref{sec:asThAndasCor} is devoted to the proof of Theorem \ref{asTh} and in Section \ref{sec:Quantum} we show Theorem \ref{QuantumTheorem}.

\section*{Acknowledgments}
The authors thank Chris Jennings-Shaffer for helpful comments on an earlier version of the paper.

\section{Proof of Theorem \ref{partial1} and Theorem \ref{ThPh}}\label{sec:proofThm:Partials}
\subsection{An auxiliary lemma}
To determine the modular transformation of $\wh{\Psi}_{Q,\b \mu}$ we use Poisson summation, for which we in particular require the explicit evaluation of a certain Fourier transform. For this, we define
\begin{align*}
F_{\tau,w}(\b  x)&\coloneqq \sqrt{i (w-\tau)} 
\erf \lp - i \sqrt{\pi  i (w- \tau)}  B(\boldsymbol{c},\boldsymbol{x}) \rp  e^{ 2\pi  i Q(\boldsymbol{x})\tau}. 
\end{align*}
Note that we may write
\begin{align}\label{rewriteF}
F_{\tau,w}(\b x)&=
2 (w-\tau) B(\boldsymbol{c},\,\boldsymbol{x}) e^{ 2\pi  i Q(\boldsymbol{x})\tau}  
\int_0^1 e^{\pi  i (w - \tau) B(\boldsymbol{c},\,\boldsymbol{x})^2 t^2 } \ddd t.
\end{align}
For a function $f:\IR^N\to\IC$, we define the \textit{Fourier transform of $f$ with respect to $Q$} as
\begin{equation*}
\mathcal F(f)(\boldsymbol{x}) \coloneqq 
| \det (A) |^{\frac{1}{2}} \int_{\IR^N} f(\boldsymbol{y}) e^{- 2 \pi i B(\boldsymbol{x},\,\boldsymbol{y})} \boldsymbol{\ddd y},
\end{equation*}
where $Q(\b n)=:\frac12\b n^TA\b n$ and $\boldsymbol{\ddd y}:= dy_1 \cdot \ \dots \ \cdot dy_N$. The following lemma shows that $F_{\tau, w}$ is basically self-dual under the Fourier transform with respect to $Q$.

\begin{lem}\label{lem:error_self_dual}
	We have
	\begin{equation*}
	\mathcal F(F_{\tau,w}) (\b x)= (-i)^{- \frac{N}{2}}  \tau^{-\frac{N-1}{2}} w^{\frac{1}{2}} 
	F _{-\frac{1}{\tau}, -\frac1w}(\b x).
	\end{equation*}
\end{lem}
\begin{proof}
	Using \eqref{rewriteF}, we obtain
	\begin{equation*}
	\mathcal F({F}_{\tau, w}) (\boldsymbol{x}) = 2 (w-\tau)
	| \det (A) |^{\frac{1}{2}} \int_{\IR^N}
	\int_0^1 B(\boldsymbol{c},\boldsymbol{y})  e^{\pi i (w-\tau) B(\boldsymbol{c},\;\boldsymbol{y})^2 t^2+2\pi i Q(\boldsymbol{y})\tau - 2 \pi i B(\boldsymbol{x},\; \boldsymbol{y})}\ddd t \boldsymbol{\ddd y}.
	\end{equation*} 
	Note that changing the order of integration is allowed since the integrals are absolutely convergent.
	
	We next switch to an orthogonal basis for $\boldsymbol{y}$ with one of the basis vectors taken to be $\boldsymbol{c}$. In particular, we write $A = C C^T$ for some $C \in \IR^{N \times N}$ with $\boldsymbol{c}^T C = \pmat{0 & \ldots & 0 & 1}$. Letting $C^T \boldsymbol{y} =: \pmat{\mathbbm{y} \\ y_0}$ and $C^T \boldsymbol{x} =: \pmat{\mathbbm{x} \\ x_0}$, with $\mathbbm{x}, \mathbbm{y} \in \IR^{N-1}$, $x_0,y_0 \in \IR$,  we obtain
	\begin{equation*}
	\mathcal F(F_{\tau,w})(\b x)= 2 (w-\tau)
	\int_{\IR^{N-1}} \int_\IR 
	\int_0^1 y_0  e^{\pi i (w-\tau) y_0^2 t^2 + \pi i y_0^2\tau + \pi i \y^T \y \tau- 2\pi i \xx^T \y - 2\pi i x_0 y_0 } 
	\ddd t\ddd y_0\boldsymbol{\ddd \y}.
	\end{equation*}
	Explicitly computing the integrals on $\y$ and $y_0$, we get 
	\begin{equation*}
	\mathcal F(F_{\tau,w})(\b x) = 2(-i)^{-\frac{1}{2}} (w-\tau)   (- i \tau)^{-\frac{N-1}{2}} 
	e^{- \frac{2\pi  i}{\tau} Q(\boldsymbol{x})}   x_0  \int_0^1 
	\frac{e^{\pi i x_0^2 \lp \frac{1}{\tau} - \frac{1}{\tau+  (w-\tau)t^2}\rp}}{\lp \tau +  (w-\tau) t^2 \rp^{\frac{3}{2}}}
	\ddd t.
	\end{equation*}
	Changing variables from $t$ to $u$ where
	\begin{equation*}
	\frac{1}{\tau} - \frac{1}{\tau+(w-\tau)t^2} = \lp \tfrac{1}{\tau} - \tfrac{1}{w} \rp u^2
	\end{equation*}
	gives the claim.
\end{proof}

\subsection{Proof of Theorem \ref{ThPh}}
To prove Theorem \ref{ThPh} we use a slightly modified function that avoids the multiplier $\chi_{\tau,w}$, namely 
\begin{equation*}
\wt{\Psi}_{Q,\boldsymbol{\mu}}( \boldsymbol{z}; \tau, w) 
\coloneqq  \sqrt{i(w-\tau)} 
\wh{\Psi}_{Q,\boldsymbol{\mu}}( \boldsymbol{z};\tau, w).
\end{equation*}
Theorem \ref{ThPh} then follows immediately from the following theorem.

\begin{thm}\label{thm:modularity_thm}
	The function $\wt{\Psi}_{Q,\b \mu}$ is a holomorphic function of $w$ (and of $\tau$ if $\boldsymbol{z} = 0$) satisfying the following Jacobi transformation properties:
	\begin{enumerate}[leftmargin=*, label={\rm (\arabic*)}]
		\item We have, for $\b m,\b r\in L$,
		\begin{align*}
		\wt{\Psi}_{Q,\boldsymbol{\mu}}(\boldsymbol{z} + \boldsymbol{m}\tau+\boldsymbol{r}; \tau, w) &= (-1)^{2Q(\boldsymbol{m}+\boldsymbol{r})}  q^{-Q(\boldsymbol{m})} 
		e^{-2\pi i B(\boldsymbol{m},\,\boldsymbol{z})} \wt{\Psi}_{Q,\boldsymbol{\mu}}(\boldsymbol{z}; \tau,w).
		\end{align*}
		\item We have
		\begin{align*}
		\wt{\Psi}_{Q,\boldsymbol{\mu}}(\boldsymbol{z}; \tau+1,w+1)  &= e^{2 \pi i Q\left(\boldsymbol{\mu}+ \frac{\boldsymbol{\ell}}{2}\right)}
		\wt{\Psi}_{Q,\boldsymbol{\mu}}(\boldsymbol{z};\tau,w),\\
		\wt{\Psi}_{Q,\boldsymbol{\mu}}\lp \tfrac{\boldsymbol{z}}{\tau}; -\tfrac{1}{\tau}, -\tfrac{1}{w} \rp  &= 
		\frac{(-i)^{\frac{N}{2}}   }{\sqrt{| L^*/ L |}} 
		\tau^{\frac{N-1}{2}} w^{-\frac{1}{2}} 
		e^{2\pi i \frac{Q(\boldsymbol{z})}{\tau}- \pi i Q(\boldsymbol{\ell})} 
		\sum_{\boldsymbol{\nu} \in L^* / L}  e^{- 2 \pi i B(\boldsymbol{\mu},\,\boldsymbol{\nu})}  
		\wt{\Psi}_{Q,\boldsymbol{\nu}} ( \boldsymbol{z}; \tau, w).
		\end{align*}
	\end{enumerate}
\end{thm}

\begin{proof}
	We first prove that the series defining $\wt{\Psi}_{Q,\boldsymbol{\mu}}$ converges to a holomorphic function of $w$. Each summand is holomorphic in $w$ (and in $\tau$ if $\boldsymbol{z}=0$) by \eqref{rewriteF}, so we prove holomorphicity of the resulting series by showing that it is absolutely and uniformly convergent on compact subsets of $\IC^N \times \IH \times \IH$. This follows by bounding (with $w_2 \coloneqq \mathrm{Im}(w)$)
	\begin{equation*}
	\left| q^{Q(\boldsymbol{n})}  e^{2 \pi i B\left(\boldsymbol{n},\,\boldsymbol{z}+\frac{\boldsymbol{\ell}}{2}\right)} \right|
	= 
	e^{-2\pi Q(\boldsymbol{n})\tau_2- 2\pi B(\boldsymbol{n},\,\Im (\boldsymbol{z}))} = 
	e^{-2\pi Q\left(\boldsymbol{n}+\frac{\Im(\boldsymbol{z})}{\tau_2}\right)\tau_2+ \frac{2\pi}{\tau_2} Q(\Im(\boldsymbol{z}))},
	\end{equation*}
	and, using \eqref{rewriteF}, we obtain 
	\begin{multline*}
	\left|   \sqrt{i(w-\tau)}  \erf \lp -i\sqrt{\pi i(w-\tau)}B\left(\boldsymbol{c}, \boldsymbol{n}+ \tfrac{\Im(\boldsymbol{z})}{\tau_2}\right) \rp  \right|   
	\\  \qquad
	\leq 2 |w -\tau| \left| B\left(\boldsymbol{c}, \boldsymbol{n}+ \tfrac{\Im(\boldsymbol{z})}{\tau_2}\right) \right|
	\begin{cases}
	1 \quad &\mbox{if } w_2 > \tau_2, \\
	e^{-\pi(w_2 - \tau_2) B\left(\boldsymbol{c},\, \boldsymbol{n}+ \frac{\Im(\boldsymbol{z})}{\tau_2}\right)^2}
	\quad &\mbox{if } w_2 \leq \tau_2.
	\end{cases}
	\end{multline*}

	\noindent (1) The claim follows directly by changing variables $\b n \mapsto \b n-\b m$.
	
	\noindent (2) The first modular transformation follows immediately.
	
	To prove the second modular transformation, we write
	\begin{equation*}
	\wt{\Psi}_{Q,\boldsymbol{\mu}} (\boldsymbol{z}; \tau,w) = e^{-2\pi i \frac{Q(\Im(\boldsymbol{z}))}{\tau_2^2}\tau} 
	\sum_{\boldsymbol{n} \in \boldsymbol{\mu} + \frac{\boldsymbol{\ell}}{2}+L}
	F_{\tau, w} \lp \boldsymbol{n} +\tfrac{\Im(\boldsymbol{z})}{\tau_2}   \rp 
	e^{2 \pi i B\left(\boldsymbol{n},-\frac{\Im(\boldsymbol{z}\overline\tau)}{\tau_2} + \frac{\boldsymbol{\ell}}{2} \rp }.
	\end{equation*}
	We then compute
	\begin{multline*}
	\wt{\Psi}_{Q,\boldsymbol{\mu}} \lp \tfrac{\boldsymbol{z}}{\tau}; -\tfrac{1}{\tau}, -\tfrac{1}{w} \rp \\= e^{ \frac{2\pi i}{\tau}\frac{Q(\Im(\boldsymbol{z} \bar{\tau}))}{ \tau_2^2} + 2\pi i B\left(\boldsymbol{\mu}+\frac{\boldsymbol{\ell}}{2},\frac{\Im(\boldsymbol{z})}{\tau_2}+\frac{\boldsymbol{\ell}}{2}\right)}
	\sum_{\boldsymbol{r} \in L}
	F_{-\frac1\tau,-\frac1w} \lp \boldsymbol{r} + \boldsymbol{\mu} + \tfrac{\boldsymbol{\ell}}{2} + \tfrac{\Im (\boldsymbol{z} \bar{\tau})}{\tau_2}\rp 
	e^{2\pi i B\left(\boldsymbol{r},\,\frac{\Im(\boldsymbol{z})}{\tau_2} + \frac{\boldsymbol{\ell}}{2}\right)}.
	\end{multline*}
	We conclude from Lemma \ref{lem:error_self_dual} and the general fact that $\mathcal F(\mathcal F(f)) (\b x) = f(-\boldsymbol{x})$ that we have
	\begin{equation}\label{eq:lemma_S_transf_result}
	F_{\tau,w}(-\b x)= (-i)^{- \frac{N}{2}}  \tau^{-\frac{N-1}{2}} w^{\frac{1}{2}} \mathcal F\left(F_{-\frac1\tau,-\frac1w}\right)(\b x).
	\end{equation}
	Moreover we use that if $g(\boldsymbol{x}) = h(\boldsymbol{x}+\boldsymbol{a})   e^{2 \pi i B(\boldsymbol{b},\boldsymbol{x})}$, then we have $\mathcal F(g)(\b y) = \mathcal F(h)(\b y-\b b)e^{2 \pi i B(\boldsymbol{a},\boldsymbol{y}-\boldsymbol{b})}$. 
	Using Poisson summation with equation \eqref{eq:lemma_S_transf_result} gives 
	\begin{multline*}
	\wt{\Psi}_{Q,\boldsymbol{\mu}}\lp \tfrac{\boldsymbol{z}}{\tau}; -\tfrac{1}{\tau}, -\tfrac{1}{w} \rp 
	=
	\frac{(-i)^{\frac{N}{2}}   \tau^{\frac{N-1}{2}}   w^{-\frac{1}{2}}}{\sqrt{| L^* / L |}} e^{ \frac{2\pi i}{\tau }\frac{Q(\Im(\boldsymbol{z} \bar{\tau}))}{\tau_2^2} + 2\pi i B\left(\boldsymbol{\mu}+\frac{\boldsymbol{\ell}}{2}, \frac{\Im(\boldsymbol{z})}{\tau_2}+\frac{\boldsymbol{\ell}}{2}\right)} 
	\\\times \sum_{\b r \in L^*}
	F_{\tau, w} \lp - \boldsymbol{r} + \tfrac{\Im(\boldsymbol{z})}{\tau_2}+ \tfrac{\boldsymbol{\ell}}{2}  \rp
	e^{2\pi i B\lp  \boldsymbol{r} - \frac{\Im(\boldsymbol{z})}{\tau_2} -  \frac{\boldsymbol{\ell}}{2} , \; \boldsymbol{\mu} + \frac{\boldsymbol{\ell}}{2}  + \frac{\Im (\boldsymbol{z} \bar{\tau})}{\tau_2} \rp}.
	\end{multline*}
	Changing $\boldsymbol{r} \mapsto -\boldsymbol{r}$ and  writing the sum over $L^*$ as $\sum_{\boldsymbol{r} \in L^*} f(\boldsymbol{r}) = \sum_{\boldsymbol{\nu} \in L^*/ L} \sum_{\boldsymbol{r} \in L} f(\boldsymbol{r}+\boldsymbol{\nu})$ we get that the sum on $\b r$ equals
	\begin{align*}
	\sum_{\boldsymbol{\nu} \in L^*/ L} \sum_{\boldsymbol{r} \in L} 
	F_{\tau, w} \lp \boldsymbol{r} + \boldsymbol{\nu} + \tfrac{\Im(\boldsymbol{z})}{\tau_2}+ \tfrac{\boldsymbol{\ell}}{2}   \rp
	e^{-2\pi i B\lp  \boldsymbol{r} + \boldsymbol{\nu} + \frac{\Im(\boldsymbol{z})}{\tau_2} +\frac{\boldsymbol{\ell}}{2},\; \boldsymbol{\mu} + \frac{\boldsymbol{\ell}}{2}  + \frac{\Im (\boldsymbol{z} \bar{\tau})}{\tau_2} \rp}.
	\end{align*}
	The claim then follows by some simplification.	
\end{proof}

\subsection{Proof of Theorem \ref{partial1}}
To state a more precise version of Theorem \ref{partial1}, we let for \\$M\coloneqq\left(\begin{smallmatrix}	a & b \\ c & d \end{smallmatrix}\right) \in \textnormal{SL}_2(\mathbb{Z})$
\begin{equation*}
\chi_{\tau, w} (M) \coloneqq  \sqrt{\tfrac{i (w - \tau )}{(c\tau+d)(cw+d)}} 
\tfrac{\sqrt{c\tau+d}\sqrt{cw+d} }{\sqrt{i(w-\tau)}}.
\end{equation*}

\begin{thm}\label{Pht}
	We have for $M =\left(\begin{smallmatrix}
	a & b \\ c & d \end{smallmatrix}\right) \in \textnormal{SL}_2(\mathbb{Z})$
	and $m,r \in \mathbb{Z}$
	\begin{align*}
	\wh{\psi}\left(\tfrac{z}{c \tau+d}; \tfrac{a \tau + b}{c \tau+d},\tfrac{a w + b}{c w +d}\right) 
	&=\chi_{\tau, w} (M) \, \nu_\eta (M)^3   (c \tau + d)^{\frac12}  
	e^{\frac{\pi i c z^2}{c\tau+d}}
	\wh{\psi}(z;\tau,w), \\
	\wh{\psi}(z+m \tau + r;\tau,w) &= (-1)^{m+r} q^{-\tfrac{m^2}{2}} \zeta^{-m}
		\wh{\psi}(z;\tau,w). 
	\end{align*}
\end{thm}
\begin{proof} 
	The claim follows from Theorem \ref{ThPh} with $L=\IZ$, $Q(n)=\frac{n^2}{2}$, $\ell=1$, $c=1$, and $\mu=0$.
\end{proof}
	
\section{Proof of Theorem \ref{asTh} and Corollary \ref{asCor}}\label{sec:asThAndasCor}

In this section we prove a Rademacher-type exact formula for unimodal sequences. For this purpose,  in Section \ref{subsNon} we derive an obstruction to modularity equation for the false theta function $\psi(\t)$ in terms of an Eichler-type integral. Then in Section \ref{subsMordell} we cast this integral into a Mordell-type form where the $\t$-dependence of the integrand is only through an exponential term. This allows us to identify the growing pieces of the obstruction to modularity near the rationals and define a ``principal part'' for this term. We bound the remaining ``non-principal'' parts in Section \ref{subsBounds} and finish in Section \ref{subsCircle} by applying the Circle Method.

\subsection{Eichler integrals and modular transformations}\label{subsNon}
Define for $\varrho\in\IQ$
\begin{equation*}
f(\tau) \coloneqq - \frac{i}{2}  \frac{\psi(\tau)}{\eta(\tau)^2} , \qquad
g(\tau) \coloneqq  \frac{1}{\eta(\tau)^2} , \quad \mathrm{and} \quad
\mathcal{E}_\varrho (\tau) \coloneqq \int_{\varrho}^{\tau+i \infty+\e} \frac{\eta (\zz)^3}{\sqrt{i (\zz - \tau)}} \ddd \zz,
\end{equation*}
where the integration path avoids the branch-cut.
To apply the Circle Method we first determine the ``false'' modular behavior of $f$.
\begin{lem}\label{4.1}
	We have, for $M=\left(\begin{smallmatrix}
	a & b \\ c & d
	\end{smallmatrix}\right)\in\SL_2(\IZ)$ with $c>0$,
	\begin{equation*}
	f (\tau) =  e^{\frac{\pi i}{4}} \nu_\eta (M)^{-1}
	\sqrt{- i (c \tau + d)}  \left(
	f \lp \tfrac{a \tau+b}{c \tau + d} \rp
	- \frac{1}{2} g \lp \tfrac{a \tau+b}{c \tau + d} \rp
	\mathcal{E}_{\frac{a}{c}} \lp \tfrac{a \tau+b}{c \tau + d} \rp
	\right).
	\end{equation*}
\end{lem}
\begin{proof}\let\qed\relax
Plugging $z=0$ into \eqref{definePh} gives 
\begin{equation*}
\wh{\psi}(\tau,w) =  i \sum_{n\in\IZ} \erf\left(-i \sqrt{\pi i(w-\tau)}\left(n+\tfrac12 \right)\right) (-1)^n   q^{\frac12\left(n+\frac12\right)^2},
\end{equation*}
where $\wh{\psi}(\tau,w)\coloneqq\wh\psi(0;\tau,w)$. Using that $\erf'(x) = \frac{2}{\sqrt{\pi}}e^{-x^2}$, we obtain, away from the branch cut
\begin{align*}
\frac{\del}{\del w} \wh{\psi}(\tau,w) &= \frac{i}{\sqrt{i (w-\tau )}}
\sum_{n \in \IZ} (-1)^n \lp n + \tfrac12 \rp e^{\pi i  \lp n + \frac12 \rp^2w} = \frac{i}{\sqrt{i (w-\tau)}}   \eta (w)^3.
\end{align*}
Using \eqref{limit}, we have 
\begin{equation}\label{PhP}
\wh{\psi}(\tau,w)  = \psi(\tau) - i \int_w^{\tau + i \infty+ \e} \frac{\eta (\zz)^3}{\sqrt{i (\zz-\tau)}} \ddd \zz,
\end{equation}
with an integration path that avoids the branch-cut.

Our next goal is to determine the  term  preventing  $\psi(\tau)$ from being a modular form. For this set $z=0$ in Theorem \ref{Pht} to obtain
\begin{equation*}
\wh{\psi}\left( \tfrac{a \tau + b}{c \tau+d},\tfrac{a w + b}{c w +d}\right) 
= \chi_{\tau, w} (M)   \nu_\eta (M)^3   (c \tau + d)^{\frac12}   \wh{\psi}(\tau,w).
\end{equation*}
Plugging \eqref{PhP} into this directly implies that
\begin{align*}
&\psi(\tau) -\chi_{\tau, w} (M)   \nu_\eta (M)^{-3}   (c \tau + d)^{-\frac12}
\psi \lp \tfrac{a \tau + b}{c \tau + d} \rp   \\
& \quad  =
i \int_w^{\tau+i \infty+\e} \frac{\eta (\zz)^3}{\sqrt{i (\zz - \tau)}} \ddd \zz 
- i \chi_{\tau,w}(M)   \nu_\eta (M)^{-3}   (c \tau + d)^{-\frac12}
\int_{\frac{aw+b}{cw+d}}^{\frac{a \tau + b}{c \tau + d}+ i \infty + \e} \frac{\eta (\zz)^3}{\sqrt{i \left(\zz - \tfrac{a \tau + b}{c \tau+d} \right)}} \ddd \zz .
\end{align*}
Now we take $w \to \tau + i \infty + \e$ in this equation and assume that $c>0$. Noting that
\begin{equation*}
%\label{limit1}
\lim_{w \to \tau + i \infty + \e} \chi_{\tau, w} (M)=1,
\end{equation*}
we get
\begin{align*}
&\psi(\tau) - 
\nu_\eta (M)^{-3}   (c \tau + d)^{-\frac12}
\psi \lp \tfrac{a \tau + b}{c \tau + d} \rp  =
- i   \nu_\eta (M)^{-3}   (c \tau + d)^{-\frac12}
\int_{\frac{a}{c}}^{\frac{a\tau+b}{c \tau +d} + i \infty + \e} \frac{\eta (\zz)^3}{\sqrt{i \left(\zz - \tfrac{a \tau + b}{c \tau+d}\right)}} \ddd \zz .
\end{align*}
The claim then follows using the transformation law,\\[5pt]
\begin{minipage}{0.95\textwidth}
\begin{equation}\label{mulet}
\eta \lp \frac{a \tau + b}{c \tau + d} \rp = \nu_\eta (M) \, (c \tau + d)^{\frac12} \, \eta(\tau).
\end{equation}
\end{minipage}
\vskip-1.5em\hspace{\textwidth} \hspace{-1cm}\qedsymbol
\vskip1em
\end{proof}

\subsection{The obstruction to modularity term as a Mordell-type integral}\label{subsMordell}
In this section we rewrite the error integrals occurring in Lemma \ref{lemInd} as Mordell-type integrals.
\begin{lem}\label{4.2}
	We have, for $\varrho\in\IQ$ and  $V\in\mathbb{C}$ with $\mathrm{Re}(V)>0$,
	\begin{equation*}
	\mathcal{E}_{\varrho} \lp \varrho + i V\rp  =  -\frac{i}{\pi}
	\sum_{n \in \IZ} (-1)^n  
	e^{\pi i \left( n + \frac{1}{2} \right)^2\varrho}  
	\lim_{\e \to 0^+}
	\int_{-\infty}^\infty \frac{e^{-\pi V x^2}}{x-\lp n+ \frac{1}{2} \rp (1+i\e)} \ddd x .
	\end{equation*}
\end{lem}

\begin{proof}
We write 
\begin{equation*}
\mathcal{E}_{\varrho} \lp\varrho + i V\rp 
 = i \int_{-V}^{\infty - i \e} \frac{\eta \lp  \varrho + i (\zz+V) \rp^3}{\sqrt{-\zz}}  \ddd \zz,
\end{equation*} 
where the path of  integration should avoid the branch-cut as in Figure \ref{fig01:integration_path}.
\begin{figure}[h!]
	\vspace{-10pt}
	\centering
	\includegraphics[scale=0.20]{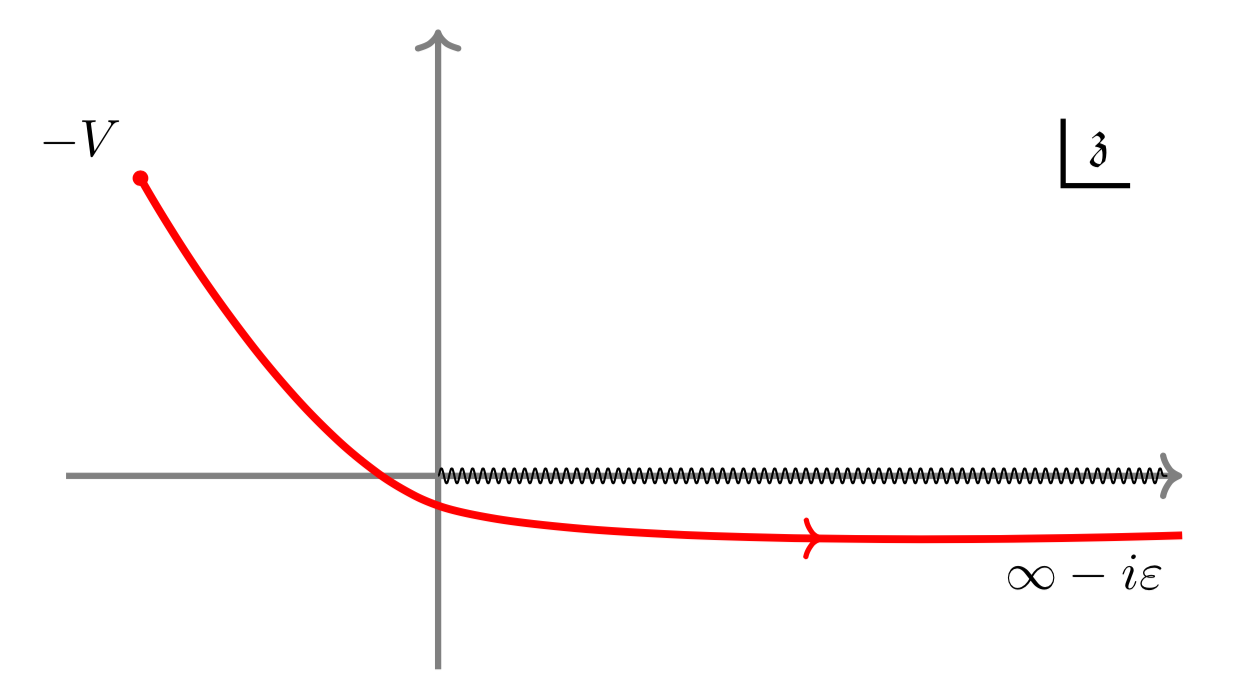}
	\vspace{-10pt}
	\caption{ Integration path and the branch-cut on the $\zz$-plane. }
	\label{fig01:integration_path}
\end{figure}

Plugging in the Fourier expansion of $\eta$, we obtain 
\begin{equation}\label{eq:error_term_series}
\mathcal{E}_{\varrho} \lp \varrho + i V\rp 
= 
i \int_{-V}^{\infty - i \e} \sum_{n \in \IZ} (-1)^n \lp n + \tfrac{1}{2} \rp
\frac{e^{- \pi \left(n + \frac{1}{2}\right)^2 (\zz+V)+\pi i \left( n + \frac{1}{2} \right)^2\varrho}}{\sqrt{-\zz}}  \ddd \zz .
\end{equation}
Recall that the exponential decay of $\eta \lp  \varrho + i (\zz+V) \rp^3$ as $\zz \to -V$ and $\zz \to \infty$ ensures that the integral is absolutely convergent. In the series expansion in \eqref{eq:error_term_series}, the individual summands are exponentially decaying as $\zz \to \infty$ but we lose that property as $\zz \to -V$. To be able to interchange the integral and the sum, we rewrite 
\begin{equation*}
\mathcal{E}_{\varrho} \lp \varrho + i V\rp 
= 
i \lim_{\d \to 0^+} 
\int_{-V+\d}^{\infty - i \e} \sum_{n \in \IZ} (-1)^n \lp n + \tfrac{1}{2} \rp
\frac{e^{- \pi \left(n + \frac{1}{2}\right)^2 (\zz+V) + \pi i \left( n + \frac{1}{2} \right)^2\varrho}}{\sqrt{-\zz}}  \ddd \zz.
\end{equation*}
We now have absolute convergence (as is clear from the explicit evaluation of the integrals below)
and can exchange the sum and the integral to obtain
\begin{equation}  \label{eq:error_limit_form}
\mathcal{E}_{\varrho} \lp \varrho + i V\rp  = 
i \lim_{\d \to 0^+} 
\sum_{n \in \IZ} (-1)^n \lp n + \tfrac{1}{2} \rp 
e^{\pi i \left( n + \frac{1}{2} \right)^2\varrho- \pi \left( n + \frac{1}{2} \right)^2V} 
\int_{-V+\d}^{\infty - i \e} 
\frac{e^{- \pi  \left( n + \frac{1}{2} \right)^2\zz}}{\sqrt{-\zz}}  \ddd \zz .
\end{equation}

To evaluate the integral we rewrite it as
\begin{align*}
\int_{-V+\d}^{\infty - i \e} 
\frac{e^{- \pi \left( n + \frac{1}{2} \right)^2 \zz}}{\sqrt{-\zz}}  \ddd \zz 
&= 
\int_{0}^{\infty - i \e} 
\frac{e^{-\pi \left( n + \frac{1}{2} \right)^2 \zz}}{\sqrt{-\zz}}  \ddd \zz 
+
\int_{-V+\d}^{0} 
\frac{e^{- \pi \left( n + \frac{1}{2} \right)^2 \zz}}{\sqrt{-\zz}}  \ddd \zz   
\notag
\\
&= -\frac{i}{n+\frac{1}{2}} 
\lp  \sgn \lp n+\tfrac{1}{2} \rp + \erf \lp i \lp n + \tfrac{1}{2} \rp \sqrt{\pi (V-\d)} \rp \rp .
\end{align*}
Plugging this into \eqref{eq:error_limit_form}, we obtain
\begin{multline}\label{Ed}
\mathcal{E}_{\varrho} \lp \varrho + i V\rp \\
 = 
\lim_{\d \to 0^+} 
\sum_{n \in \IZ} (-1)^n  
e^{\pi i \left( n + \frac{1}{2} \right)^2\varrho - \pi \left( n + \frac{1}{2} \right)^2V} \lp  \sgn \lp n+\tfrac{1}{2} \rp + \erf \lp i \lp n + \tfrac{1}{2} \rp \sqrt{\pi (V-\d)} \rp \rp .
\end{multline}
Using that $\erf (i z) = \frac{2i}{\sqrt{\pi}} \int_0^z e^{t^2} \ddd t$, yields the asymptotic behavior 
\begin{equation*}
\erf (iz) = \frac{ie^{z^2}}{\sqrt{\pi}z} \lp 1 + O \lp |z|^{-2} \rp \rp ,
\end{equation*}
if $| \mathrm{Arg} (\pm z) | < \frac{\pi}{4}$ as $|z| \to \infty$. Therefore, we need to carefully take $\d \to 0^+$ as the expression at $\d = 0$ is not absolutely convergent. Separating this main term of the error function as
\begin{equation}\label{split}
\lp \erf \lp i \lp n + \tfrac{1}{2} \rp \sqrt{\pi (V-\d)} \rp 
- \frac{i   e^{\pi \left( n + \tfrac{1}{2} \right)^2 (V-\d)} }{\pi \lp n+\tfrac{1}{2} \rp \sqrt{V-\d}} \rp 
+ \frac{i   e^{\pi \left( n + \tfrac{1}{2} \right)^2 (V-\d)} }{\pi \lp n+\tfrac{1}{2} \rp \sqrt{V-\d}}  ,
\end{equation}
we find that everything except for the contribution from this last term is absolutely and uniformly convergent on compact subsets of $\mathrm{Re} (V) > 0$ and $0 \leq \d \leq \d_0$ for sufficiently small $\d_0$. That means we can plug in $\d = 0$ for these terms to take the limit. 

Thus we focus on the last term in \eqref{split} whose contribution to $\mathcal{E}_{\varrho} \lp \varrho + i V\rp$ is 
\begin{equation*}
\lim_{\d \to 0^+}  \frac{i}{\pi \sqrt{V-\d}}
\sum_{n \in \IZ}  
\frac{(-1)^n  e^{\pi i \left( n + \frac{1}{2} \right)^2\varrho}}{n+\frac{1}{2}}
e^{- \pi \left( n + \frac{1}{2} \right)^2 \d}  .
\end{equation*}
The series is absolutely convergent for any $\mathrm{\d} >0$. If the corresponding series is also convergent for $\d =0$, then by Abel's Theorem (viewing it as a power series in $e^{-\pi \d}$), the limit as $\d \to 0^+$ is simply the value at $\d = 0$. To prove convergence at $\d = 0$, let $\varrho = \frac{h}{k}$ with $\gcd (h,k) = 1$ and $k > 0$ and consider for $\nu_1,\nu_2 \in \IN$ the following sum
\begin{equation*}
\sum_{-\nu_1 \leq n \leq \nu_2}  
\frac{(-1)^n   e^{\pi i \left( n + \frac{1}{2} \right)^2\frac{h}{k}}}{n+\frac{1}{2}}
= \frac{k}{2} \sum_{-\nu_1 \leq n \leq \nu_2}
\frac{(-1)^n    e^{\pi i \left( n + \frac{1}{2} \right)^2\frac{h}{k}} }
{ \lp n+\tfrac{1}{2} \rp \lp n+k+\tfrac{1}{2} \rp  }+O\left(\tfrac{k}{\nu_1}+\tfrac{k}{\nu_2}\right).
\end{equation*}
The sum on the right-hand side is absolutely convergent as $\nu_1, \nu_2 \to \infty$.

We can now set $\d = 0$ in \eqref{Ed} and hence obtain
\begin{equation}\label{eq:modular_error_erf}
\mathcal{E}_{\varrho} \lp \varrho + i V\rp  = 
\sum_{n \in \IZ} (-1)^n  
e^{\pi i \left( n + \frac{1}{2} \right)^2\varrho - \pi \left( n + \frac{1}{2} \right)^2V} 
\lp  \sgn \lp n+\tfrac{1}{2} \rp + \erf \lp i \lp n + \tfrac{1}{2} \rp \sqrt{\pi V} \rp \rp .
\end{equation}

Next we use the following identity, which is valid for $ s \in \IR \setminus \{ 0 \}$ and $\mathrm{Re} (V) > 0$,
\begin{equation}\label{eq:erf_integral_form}
e^{- \pi s^2 V} \lp \sgn (s) + \erf \lp i s \sqrt{\pi V} \rp \rp 
=
-\frac{i}{\pi}
\lim_{\e \to 0^+}
\int_{-\infty}^\infty \frac{e^{-\pi V x^2}}{x-s(1+i\e)} \ddd x.
\end{equation}
Plugging \eqref{eq:erf_integral_form} into \eqref{eq:modular_error_erf} gives the claim.
\end{proof}

\subsection{Bounds on the Mordell-type integrals}\label{subsBounds}
Let $\varrho = \frac{h'}{k}$ where $h',k$ are integers satisfying $\mathrm{gcd}(h',k) = 1$ and $k>0$. Given a real number $d$ with $0 \leq d < \frac{1}{8}$ we split the obstruction to modularity from Lemma \ref{4.1} (as rewritten in Lemma \ref{4.2}) as follows
\begin{equation*}
e^{2 \pi d V}  \mathcal{E}_{\frac{h'}{k} } \lp \tfrac{h'}{k} + i V\rp  
=
\mathcal{E}^*_{\frac{h'}{k},d} \lp \tfrac{h'}{k} + i V\rp  
+
\mathcal{E}^{\rm e}_{\frac{h'}{k},d} \lp \tfrac{h'}{k} + i V\rp ,
\end{equation*}
where
\begin{align}\label{4.14}
\mathcal{E}^*_{\frac{h'}{k},d} \lp \tfrac{h'}{k} + i V\rp
&:= 
-\frac{i}{\pi}   e^{2 \pi d V} 
\sum_{n \in \IZ} (-1)^n  
e^{\pi i \left( n + \frac{1}{2} \right)^2 \frac{h'}{k}}  
\lim_{\e \to 0^+}
\int\displaylimits_{-\sqrt{2d}}^{\sqrt{2d}} \frac{e^{-\pi V x^2}}{x-\lp n+ \tfrac{1}{2} \rp (1+i\e)} \ddd x ,\\
\notag
\mathcal{E}^{\rm e}_{\frac{h'}{k},d} \lp \tfrac{h'}{k} + i V\rp
&:= 
-\frac{i}{\pi}   e^{2 \pi d V} 
\sum_{n \in \IZ} (-1)^n  
e^{\pi i \left( n + \frac{1}{2} \right)^2\frac{h'}{k}}  
\lim_{\e \to 0^+}
\int\displaylimits_{|x| \geq \sqrt{2d} } \frac{e^{-\pi V x^2}}{x-\lp n+ \tfrac{1}{2} \rp (1+i\e)} \ddd x .
\end{align}
In the following lemma, we  bound $\mathcal{E}^{\rm e}_{\frac{h'}{k},d} (\frac{h'}{k} + i V)$, by which we also prove its convergence and that of $\mathcal{E}^*_{\frac{h'}{k},d} (\frac{h'}{k} + i V)$, making the splitting of $e^{2 \pi d V}  \mathcal{E}_{\frac{h'}{k} } (\frac{h'}{k} + i V)$ justified.

\begin{lem}\label{lemInd}
	For $0 \leq d < \frac{1}{8}$, $h',k \in \IZ$, $\mathrm{gcd}(h',k) = 1$, $k>0$, and $\mathrm{Re} (V) \geq 1$ we have
	\begin{equation*}
	\mathcal{E}^{\rm e}_{\frac{h'}{k},d} \lp \tfrac{h'}{k} + i V\rp = O \left(\log (k)\right),
	\end{equation*}
	where the bound is independent of $h'$ and $V$.
\end{lem}
\begin{proof}
	We start by combining the integral over negative reals and positive reals to obtain
	\begin{equation*}\begin{split}
	&\mathcal{E}^{\rm e}_{\frac{h'}{k},d} \lp \tfrac{h'}{k} + i V\rp
	= 
	-\frac{2i}{\pi}   e^{2 \pi d V} 
	\sum_{n \in \IZ} (-1)^n  \lp n +\tfrac{1}{2} \rp
	e^{\pi i \left( n + \frac{1}{2} \right)^2\frac{h'}{k}}  
	\lim_{\e \to 0^+}
	\int\displaylimits_{\sqrt{2d}}^\infty 
	\frac{e^{-\pi V x^2}}{x^2-\lp n+ \frac{1}{2} \rp^2 (1+i\e)^2} \ddd x
	\\
	&=-\frac{i}{\pi} 
	\sum_{n \in \IZ} (-1)^n  \lp n +\tfrac{1}{2} \rp
	e^{\pi i \left( n + \frac{1}{2} \right)^2\frac{h'}{k}}  
	\lim_{\e \to 0^+}
	\int_0^\infty 
	\frac{e^{-\pi V u}}{\sqrt{u+2d} \left(u + 2d-\lp n+ \frac{1}{2} \rp^2 (1+i\e)^2\right)}  \ddd u,
	\end{split}\end{equation*}
	changing variables $u \coloneqq x^2 - 2 d$.
	Now we write
	\begin{align*}
	&\frac{1}{u + 2d-\lp n+ \frac{1}{2} \rp^2 (1+i\e)^2} 
	\\ &\qquad \qquad \qquad
	= 
	\lp \frac{1}{u + 2d-\lp n+ \frac{1}{2} \rp^2 (1+i\e)^2} 
	+ \frac{1}{ \lp n+ \frac{1}{2} \rp^2 (1+i\e)^2} \rp
	- \frac{1}{ \lp n+ \frac{1}{2} \rp^2 (1+i\e)^2}
	\\ &\qquad \qquad \qquad
	= \frac{u+2d}{\lp n+ \frac{1}{2} \rp^2 (1+i\e)^2 \left(u + 2d-\lp n+ \frac{1}{2} \rp^2 (1+i\e)^2\right)} 
	- \frac{1}{ \lp n+ \frac{1}{2} \rp^2 (1+i\e)^2},
	\end{align*}
	and consider the contribution of each term, which we denote by $\mathcal{E}^{\rm e}_{\frac{h'}{k},d,1} (\frac{h'}{k} + i V)$ and $\mathcal{E}^{\rm e}_{\frac{h'}{k},d,2}(\frac{h'}{k} + i V)$, separately. We start with $\mathcal{E}^{\rm e}_{\frac{h'}{k},d,1} (\frac{h'}{k} + i V)$ and write
	\begin{equation*}
	\mathcal{E}^{\rm e}_{\frac{h'}{k},d,1} \lp \tfrac{h'}{k} + i V\rp
	=
	-\frac{i}{\pi} 
	\sum_{n \in \IZ} 
	\frac{(-1)^n  e^{\pi i \left( n + \frac{1}{2} \right)^2\frac{h'}{k}} }{n +\frac{1}{2}} 
	\lim_{\e \to 0^+}
	\int_0^\infty 
	\frac{e^{-\pi V u}    \sqrt{u+2d}}{u + 2d-\lp n+ \frac{1}{2} \rp^2 (1+i\e)^2}   \ddd u.
	\end{equation*}
	Note that because $\mathrm{Re} (V) \geq 1$ we have either $\mathrm{Re} ( V e^{\frac{\pi i}{4}}) \geq \frac{1}{\sqrt{2}}$ or $\mathrm{Re} ( V e^{-\frac{\pi i}{4}} ) \geq \frac{1}{\sqrt{2}}$. 
	Using Cauchy's Theorem we rotate the path of integration to $e^{\frac{\pi i}{4}} \IR^+$ if $\mathrm{Re} (V e^{\frac{\pi i}{4}}) \geq \frac{1}{\sqrt{2}}$ or to $e^{-\frac{\pi i}{4}} \IR^+$ if $\mathrm{Re} (Ve^{-\frac{\pi i}{4}}) \geq \frac{1}{\sqrt{2}}$, picking up residues from poles (which lie just above the real line) in the former case.
	
	The contribution from the poles at $-2 d + ( n+ \frac{1}{2} )^2 (1+i\e)^2$ sums to 
	\begin{equation*}
	2 \sum_{n \in \IZ} 
	(-1)^n 
	\sgn\lp n + \tfrac{1}{2} \rp e^{\pi i \left( n + \tfrac{1}{2} \right)^2\frac{h'}{k}- \pi  \lp \left(n+ \frac{1}{2}\right)^2 - 2d  \rp V}.
	\end{equation*}
	Since $d < \frac{1}{8}$ and $\mathrm{Re}(V) \geq 1$ we can bound the absolute value of this contribution against
	\begin{equation*}
	2 \sum_{n \in \IZ} e^{- \pi \lp \left(n+ \frac{1}{2}\right)^2 - 2d  \rp} ,
	\end{equation*}
	which is a constant independent of $k, h',V$ and thus the poles at most contribute $O(1)$.
	
	We next bound the remaining part and distinguish whether $\mathrm{Re} ( V  e^{\frac{\pi i}{4}} ) \geq \frac{1}{\sqrt{2}}$ or $\mathrm{Re} ( V  e^{-\frac{\pi i}{4}} ) \geq \frac{1}{\sqrt{2}}$. We first assume that $\mathrm{Re} ( V  e^{\frac{\pi i}{4}} ) \geq \frac{1}{\sqrt{2}}$ in which case we rotate the path of integration to $e^{\frac{\pi i}{4}} \IR^+$.
	Now, the poles are away from the path of integration and hence the integral is holomorphic in $\e$ around zero, so we can set $\e = 0$ to take the limit (also changing variables $u \mapsto e^{\frac{\pi i}{4}} u$) to obtain
	\begin{equation*}
	\mathcal{I}_{\frac{h'}{k},d} \lp \tfrac{h'}{k} + i V\rp
	\coloneqq
	-\frac{i}{\pi}  e^{\frac{\pi i}{4}}
	\sum_{n \in \IZ} 
	\frac{(-1)^n  e^{\pi i \left( n + \frac{1}{2} \right)^2\frac{h'}{k}} }{n +\frac{1}{2}} 
	\int_0^\infty
	\frac{e^{-\pi V \frac{1+i}{\sqrt{2}} u}    \sqrt{u\frac{1+i}{\sqrt{2}}+2d}}{u \frac{1+i}{\sqrt{2}} + 2d-\lp n+ \tfrac{1}{2} \rp^2}   \ddd u.
	\end{equation*}
	Using the bounds
	\begin{equation*}
	\left| \displaystyle e^{-\pi V \frac{1+i}{\sqrt{2}} u} \right| \leq
	e^{-\frac{\pi u}{\sqrt{2}}} \qquad \mbox{and} \qquad
	\frac{1}{\left| u \frac{1+i}{\sqrt{2}} + 2d-\lp n+ \tfrac{1}{2} \rp^2 \right|} \leq 
	\frac{\sqrt{2}}{\lp n+ \tfrac{1}{2} \rp^2 -2d},
	\end{equation*}
	yields
	\begin{equation*}
	\left| \mathcal{I}_{\frac{h'}{k},d} \lp \tfrac{h'}{k} + i V\rp \right| \leq \frac{\sqrt{2}}{\pi}
	\sum_{n \in \IZ} \frac{1}{\left| n + \frac{1}{2} \right|\left(\lp n+ \frac{1}{2} \rp^2 -2d\right)}
	\int_0^\infty  e^{-\frac{\pi u}{\sqrt{2}}} \left|  u\tfrac{1+i}{\sqrt{2}}+2d  \right|^{\frac{1}{2}} \ddd u \ll 1,
	\end{equation*}
	with the upper bound independent of $k,h',V$.
	
	If $\mathrm{Re} ( V  e^{-\frac{\pi i}{4}} ) \geq \frac{1}{\sqrt{2}}$, exactly the same argument with the path of integration rotated to $e^{-\frac{\pi i}{4}} \IR^+$ shows all in all we have
	\begin{equation*}
	\mathcal{E}^{\rm e}_{\frac{h'}{k},d,1} \lp \tfrac{h'}{k} + i V\rp = O (1),
	\end{equation*}
	with the upper bound independent of $k,h',V$.

	Next, we investigate $\mathcal{E}^{\rm e}_{\frac{h'}{k},d,2} (\frac{h'}{k} + i V)$, for which we take $\varepsilon \to 0^+$ to obtain
	\begin{equation*}
	\mathcal{E}^{\rm e}_{\frac{h'}{k},d,2} \lp \tfrac{h'}{k} + i V\rp = 
	\frac{i}{\pi} 
	\sum_{n \in \IZ}
	\frac{ (-1)^n  e^{\pi i \left( n + \frac{1}{2} \right)^2\frac{h'}{k}}}{n  + \frac{1}{2} }  
	\int_0^\infty 
	\frac{e^{-\pi V u}}{\sqrt{u+2d}}  \ddd u.
	\end{equation*}
	We estimate
	\begin{equation*}
	\left| \int_0^\infty 
	\frac{e^{-\pi V u}}{\sqrt{u+2d}}  \ddd u \right| \leq
	\int_0^\infty 
	\frac{e^{-\pi u}}{\sqrt{u+2d}}  \ddd u  \ll 1.
	\end{equation*}
	So we only need to bound the series in $n$. We start by writing
	\begin{equation*}
	\mathcal{I}_{h',k} \coloneqq \sum_{n \in \IZ}
	\frac{ (-1)^n  e^{\pi i \left( n + \frac{1}{2} \right)^2\frac{h'}{k}}}{n  + \frac{1}{2} } 
	=
	\lim_{N \to \infty} \sum_{m=-N}^N \ \sum_{r \pmod{2k}}
	\frac{(-1)^r  e^{\pi i \left( r + \frac{1}{2} \right)^2\frac{h'}{k}}}{2km + r + \frac{1}{2}},
	\end{equation*}
	where we note that for $n = 2km + r$ with $m, r \in \IZ$ we have
	\begin{equation*}
	(-1)^n e^{\pi i \lp n + \frac{1}{2} \rp^2\frac{h'}{k}} 
	=  (-1)^r  e^{\pi i \lp r + \frac{1}{2} \rp^2\frac{h'}{k}}.
	\end{equation*}
	Moving the finite sum outside (which we can because the resulting series in $m$ converges) gives
	\begin{equation*}
	\mathcal{I}_{h',k} = \frac{1}{2k} \sum_{r \pmod{2k}} 
	(-1)^r  e^{\pi i \left( r + \frac{1}{2} \right)^2\frac{h'}{k}}
	\lim_{N \to \infty} \sum_{m=-N}^N \frac{1}{m + \frac{r + \frac{1}{2}}{2k}}.
	\end{equation*}
	Now, we use the identity
	\begin{equation}\label{eq:cot_identity}
	\pi \cot (\pi x) = \lim_{N\to \infty} \sum_{n=-N}^N \frac{1}{x+n},
	\end{equation}
	to get
	\begin{equation*}
	\mathcal{I}_{h',k} = \frac{\pi}{2k} \sum_{r \pmod{2k}} 
	(-1)^r  e^{\pi i \left( r + \frac{1}{2} \right)^2\frac{h'}{k}}
	\cot \lp \tfrac{\pi}{2k} \lp r + \tfrac{1}{2} \rp  \rp.
	\end{equation*}
	Noting that for $0<x<1$ we have $| \cot (\pi x) | \ll \frac{1}{x} + \frac{1}{1-x}$,
	we obtain
	\begin{equation*}
	\mathcal{I}_{h',k} \ll \frac{1}{k} \sum_{r =0}^{2k-1} 
	\frac{2k}{r+\frac{1}{2}} = O \left(\log (k)\right).
	\end{equation*}
	This finishes the proof.
\end{proof}
\begin{rem}
Lemma \ref{lemInd} with $d=0$ implies that the obstruction to modularity term is bounded as
\begin{equation*}
 \mathcal{E}_{\frac{h'}{k} } \lp \tfrac{h'}{k} + i V\rp   =  O \left(\log (k)\right),
\end{equation*}
for $h',k,V$ satisfying the conditions of the lemma and with the bound independent of $h'$ and $V$.
\end{rem}
\begin{rem}
Since in \eqref{4.14} we have $d < \frac{1}{8}$, the poles are away from the path of integration and we can set $\e = 0$. Moreover, replacing $n = 2km + r$ with $m\in \IZ$ and $r$ running $\!\!\pmod{2k}$ while taking the sum over $n$ in a symmetric fashion, we have
\begin{equation*}
\mathcal{E}^*_{\frac{h'}{k},d} \lp \tfrac{h'}{k} + i V\rp
= 
-\frac{i}{\pi}   e^{2 \pi d V} 
\sum_{r \pmod{2k}} 
(-1)^r  e^{\pi i \left( r + \frac{1}{2} \right)^2\frac{h'}{k}}
\lim_{N \to \infty} \sum_{m=-N}^N 
\int\displaylimits_{-\sqrt{2d}}^{\sqrt{2d}} \frac{e^{-\pi V x^2}}{x-\lp 2k m  + r + \frac{1}{2} \rp } \ddd x .
\end{equation*}
We can switch the order of the integral and the sum over $m$ and use \eqref{eq:cot_identity} (note that the convergence is uniform in our finite range) to obtain
\begin{equation}\label{eq:epsilon_remaining}
\mathcal{E}^*_{\frac{h'}{k},d} \lp \tfrac{h'}{k} + i V\rp
= 
-\frac{i    e^{2 \pi d V} }{2k} 
\sum_{r \pmod{2k}} 
(-1)^r  e^{\pi i \left( r + \frac{1}{2} \right)^2\frac{h'}{k}}
\int\displaylimits_{-\sqrt{2d}}^{\sqrt{2d}}  
\cot \lp \tfrac{\pi}{2k} \lp x - r - \tfrac{1}{2} \rp \rp   e^{-\pi V x^2}
\ddd x .
\end{equation}
\end{rem}

\subsection{Applying the Circle Method}\label{subsCircle}
In this section we finish the proof of Theorem \ref{asTh} by using the Circle Method. Write
\begin{equation*}
f (\tau) = q^{\frac{1}{24}} \sum_{n\geq 0} \a_f (n) q^n = q^{\frac{1}{24}}(1+O(q))
\andd
g (\tau) = q^{-\frac{1}{12}} \sum_{n\geq 0} \a_g (n) q^n = q^{-\frac{1}{12}}(1+O(q)).
\end{equation*}
Then, by \eqref{splitU}, we have
\begin{equation*} 
u(n)=\a_g(n)-\a_f(n).
\end{equation*}
Applying the Circle Method to $g$ is standard and leads to the first term in Theorem \ref{asTh}.

To find an exact formula for $\a_f (n)$ we start by writing
\begin{equation*}
\a_f (n) = \int_{i}^{i+1}  f(\tau) e^{-2\pi i \left(n+ \frac{1}{24}\right)\tau } \ddd \tau.
\end{equation*}
where the integral goes along any path connecting $i$ and $i+1$. We decompose the integral into arcs lying near roots of unity $\zeta_k^h$, where $0\leq h\leq k\leq N$ with $\mathrm{gcd}(h,k)=1$, and $N\in\mathbb N$ is a parameter, which then tends to infinity. For this, the {\it Ford Circle} $\mathcal C_{h,k}$ denotes the circle in the complex $\tau$-plane with radius
$\frac{1}{2k^2}$ and center $\frac{h}{k}+\frac{i}{2k^2}$.
We let $P_N := \bigcup_{\frac{h}{k} \in F_N} C_{h,k}(N)$, where $F_N$ is the Farey sequence of order $N$ and $C_{h,k}(N)$ is the upper arc of the Ford Circle $\mathcal C_{h,k}$ from its intersection with $\mathcal C_{h_1,k_1}$ to its intersection with $\mathcal C_{h_2,k_2}$, where $\frac{h_1}{k_1} < \frac{h}{k} < \frac{h_2}{k_2}$ are consecutive fractions in $F_N$. In particular $C_{0,1}(N)$ and $C_{1,1}(N)$ are half-arcs with the former starting at $i$ and the latter ending at $i+1$. 
\begin{figure}[h!]
	\vspace{-10pt}
	\centering
	\includegraphics[scale=0.32]{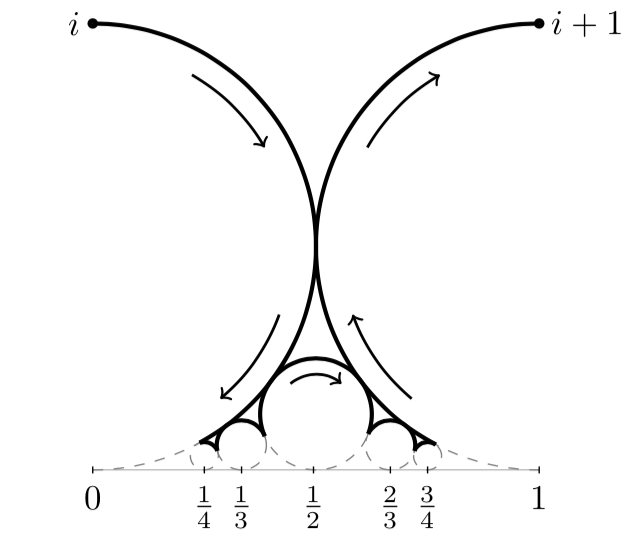}
	\vspace{-10pt}
	\caption{Rademacher's integration path $P_N$ for $N=4$. }
	\label{fig02:Rademacher_path}
\end{figure}

Therefore,
\begin{equation*}
\a_f (n) = \sum_{k=1}^N \sum_{\substack{0 \leq h \leq k \\ \mathrm{gcd} (h,k) = 1}} \ 
\int\limits_{C_{h,k}(N)}  f(\tau) e^{-2\pi i \left(n+ \frac{1}{24}\right)\tau } \ddd \tau.
\end{equation*}
Next, we make the change of variables $\tau=\frac{h}{k}+\frac{iZ}{k^2}$, which maps the Ford Circles to a standard circle with radius $\frac12$ which is centered at $Z=\frac{1}{2}$ (see Figure \ref{fig03:z_Rademacher_path}). 
The image of the arc $C_{h,k}(N)$ is now an arc on the standard circle from $Z_1$ to $Z_2$, where
\begin{equation*}
Z_1 = Z_1 (h,k; N) \coloneqq \frac{k}{k - i k_1} 
\andd
Z_2 = Z_2 (h,k; N) \coloneqq \frac{k}{k + i k_2}. 
\end{equation*}
We also combine the half-arcs $C_{0,1}(N)$ and $C_{1,1}(N)$ into an arc in the $z$-plane from $z_1 (0,1;N) \coloneqq \frac{1}{1-iN}$ to $z_2 (0,1;N) \coloneqq \frac{1}{1+iN}$ by shifting the $C_{1,1}(N)$ half-arc as $\t \mapsto \t-1$. Note that on the disc bounded by the standard circle we always have $\text{Re}(\frac1Z)\geq 1$. Moreover, for any point $z$ on the chord that is connecting $z_1$ and $z_2$, we have $|Z| \leq \frac{ \sqrt{2}k}{N}$ and the length of this chord does not exceed $\frac{2 \sqrt{2}k}{N}$.
\begin{figure}[h!]
	\vspace{-10pt}
	\centering
	\includegraphics[scale=0.22]{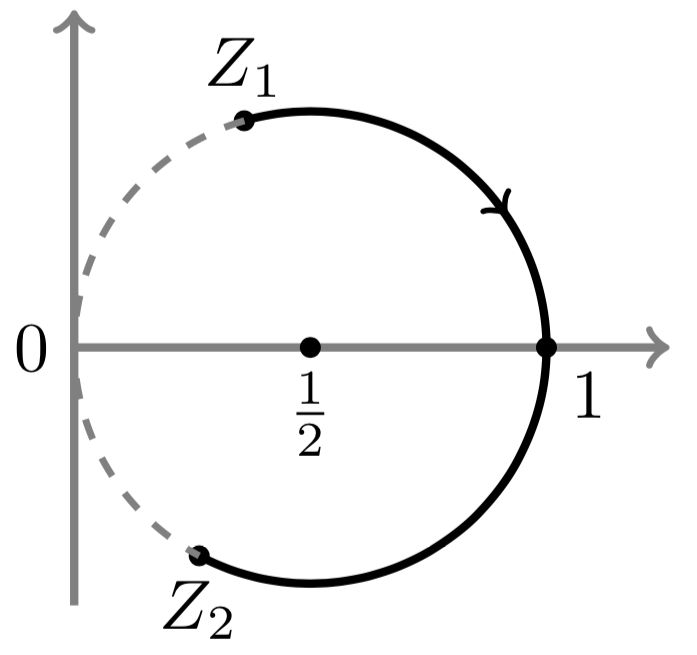}
	\vspace{-10pt}
	\caption{ Integration path on the standard circle. }
	\label{fig03:z_Rademacher_path}
\end{figure}

This yields
\begin{equation*}
\a_f (n) = i \sum_{k=1}^N k^{-2} \sum_{\substack{0 \leq h < k \\ \mathrm{gcd}(h,k) = 1}} \ \int\limits_{Z_1}^{Z_2} 
f \lp \tfrac{h}{k} + \tfrac{iZ}{k^2} \rp  e^{-2\pi i \left(n+\frac{1}{24}\right)\left(\frac{h}{k} + \frac{iZ}{k^2}\right)}  \ddd Z .
\end{equation*}
Since as $N \to \infty$ the end-points of the integration path get closer to the origin (in the $\tau$-plane, the integration path gets closer to the rationals numbers), we use the modular transformation $M_{h,k}$
and have better control on the integrand's behavior near rational points. 
Lemma \ref{4.1} implies
\begin{equation*}
f\left(\tfrac{h}{k}+\tfrac{iZ}{k^2}\right)= e^{\frac{\pi i}{4}} \nu_\eta(M_{h,k})^{-1} \sqrt{\tfrac{Z}{k}} \left(f\left(\tfrac{h'}{k}+\tfrac{i}{Z}\right)-\tfrac12 g\left(\tfrac{h'}{k}+\tfrac{i}{Z}\right) \mathcal E_{\tfrac{h'}{k}}\left(\tfrac{h'}{k}+\tfrac{i}{Z}\right)\right).
\end{equation*}
We then approximate the term in the parentheses by its principal part 
\begin{equation*}
-\frac{1}{2} \zeta^{-h'}_{12k}  
\mathcal{E}^*_{\frac{h'}{k},\frac{1}{12}} \lp \tfrac{h'}{k} + \tfrac{i}{Z} \rp.
\end{equation*}
The error to $\a_f (n)$ that arises from this approximation can be estimated by deforming the integral for the remainder term to the chord connecting $Z_1$ to $Z_2$ using Cauchy's Theorem. Then the bound in Lemma \ref{lemInd} with $d=0$ and $d=\frac{1}{12}$, together with the fact that $g(\t)  - q^{-\frac{1}{12}}$ and $f(\t)$ are $ O(1)$ for $\t_2 \geq 1$ uniformly in $\t$, yields
\begin{align*}
\alpha_f(n)=-\frac{ i }{2} \sum_{k=1}^N k^{-\frac{5}{2}} 
\sum_{\substack{0 \leq h < k \\ \mathrm{gcd}(h,k) = 1}}
e^{\frac{\pi i}{4}} \nu_\eta (M_{h,k})^{-1}  
\zeta_{24k}^{-(24n+1)h-2h'}
\int\limits_{Z_1}^{Z_2}  Z^{\frac12}  
\mathcal{E}^*_{\frac{h'}{k},\frac{1}{12}} &\lp \tfrac{h'}{k} + \tfrac{i}{Z} \rp 
e^{2\pi \left(n+\frac{1}{24}\right) \frac{Z}{k^2}}    \ddd Z \\
&\quad+O\left(N^{-\frac12} \log(N) \right).
\end{align*} 
We next separate the integral as 
\begin{equation*}
\int\limits_{Z_1}^{Z_2} = \int_\CC - \int_0^{Z_1} - \int_{Z_2}^0,
\end{equation*}
where the integration over the whole standard circle $\CC$ is oriented clockwise and we assume that both of the remaining two integrals are also over arcs on the standard circle. Over the standard circle we have $\mathrm{Re}(\frac{1}{Z}) = 1$ and hence $\mathcal{E}^*_{\frac{h'}{k},\frac{1}{12}} (\frac{h'}{k} + \frac{i}{Z})$ is also bounded by $\log (k)$.
Since the arc lengths and the value of $|Z|$ on the arcs from $0$ to $Z_{1}$ or $Z_2$ are also bounded by $\frac{k}{N}$, the same kind of bound used for the error terms hold and the contribution to $\a_f (n)$  from such integrals are bounded by $N^{-\frac12} \log(N)$ as well. Plugging in the expression for $\CE^*$ in equation \eqref{eq:epsilon_remaining} we then obtain
\begin{align} 
&\a_f (n) = -\frac{ 1 }{4} \sum_{k=1}^N k^{-\frac{7}{2}} 
\sum_{\substack{0 \leq h < k \\ \mathrm{gcd}(h,k) = 1}}
e^{\frac{\pi i}{4}} \nu_\eta (M_{h,k})^{-1}  
\zeta_{24k}^{-(24n+1)h-2h'}
\sum_{r \pmod{2k}} (-1)^r e^{\pi i \left(r+\frac{1}{2}\right)^2 \frac{h'}{k} }  \nonumber \\
&
\qquad\times
\int_{-\frac{1}{\sqrt{6}}}^{\frac{1}{\sqrt{6}}}  
\cot \lp \tfrac{\pi}{2k} \lp x-r-\tfrac{1}{2} \rp \rp 
\int_{\mathcal{C}}  Z^{\frac{1}{2}}   e^{2\pi  \left(n+\frac{1}{24}\right)\frac{Z}{k^2} + 2 \pi \left(\frac{1}{12} - \frac{x^2}{2} \right) \frac{1}{Z}}  \ddd Z   \ddd x + O \left(N^{-\frac12} \log(N) \right) .\label{formulaa}
\end{align}

\noindent We  perform the integrals in $Z$ using that for $\delta_1,\delta_2\in\IR^+$,
\begin{equation*}
\int_{\mathcal{C}}  Z^{\frac{1}{2}} e^{\frac{2 \pi \d_1}{Z} + 2 \pi \d_2 Z}\ddd Z 
= 
- 2 \pi i \lp \tfrac{\d_1}{\d_2} \rp^{\frac{3}{4}}
I_{\frac{3}{2}} \lp 4\pi \sqrt{\d_1 \d_2}  \rp.
\end{equation*}
Taking $N \to \infty$ in \eqref{formulaa} gives 
\begin{align}
\a_f (n) &= 
\frac{\pi i}{2^{\frac14}(24n+1)^{\frac34}}
\sum_{k\geq 1} k^{-2}
\!\!\!\sum_{\substack{0 \leq h < k \\ \mathrm{gcd}(h,k) = 1}} \!\!\! 
e^{\frac{\pi i}{4}} \nu_\eta (M_{h,k})^{-1}   \zeta_{24k}^{-(24n+1)h-2h'}
\!\!\!\sum_{r \pmod{2k}}  \!\!\!(-1)^r  e^{\pi i \left( r + \frac{1}{2} \right)^2\frac{h'}{k}}  \notag \\
&\qquad \times
\int_{-\frac{1}{\sqrt{6}}}^{\frac{1}{\sqrt{6}}}   \left(1 - 6 x^2\right)^{\frac{3}{4}}
\cot \lp \tfrac{\pi}{2k} \lp x-r-\tfrac{1}{2} \rp \rp
I_{\frac{3}{2}} \lp \tfrac{2\pi}{\sqrt{3}k} \sqrt{\left(1-6x^2\right) \lp n + \tfrac{1}{24} \rp}  \rp 
\ddd x .
\label{eq:f_final_result}
\end{align}
Plugging in the definition of $K_k(n,r)$ and changing variables then gives the claim.

\subsection{Proof of Corollary \ref{asCor}}
We are now ready to prove Corollary \ref{asCor}.
\begin{proof}[Proof of  Corollary \ref{asCor}]
	We use the following asymptotic behavior of the Bessel function as $x\to\infty$
\begin{equation*}
I_\ell(x) = \frac{e^x}{\sqrt{2\pi x}} \lp 1 +O \lp x^{-1} \rp \rp .
\end{equation*}
The leading exponential term for $\alpha_g (n)$  then comes from the $k=1$ contribution
\begin{equation*}
\frac{2\pi}{12n-1} K_1(n) I_2\left(\frac{\pi}{3}\sqrt{12n-1}\right).
\end{equation*}
Using that $K_1(n)=1$ we obtain 
\begin{equation*}
\a_g(n)=\frac{1}{4\cdot3^{\frac 3 4}  n^{\frac54}} e^{2\pi\sqrt{\frac n3}} \lp 1 +O \lp n^{-\frac{1}{2}} \rp \rp.
\end{equation*}
Similarly the leading exponential term for $\a_f(n)$ comes from the $k=1$ contribution. Using that $K_1(n, r)=(-1)^{r+1}$, we write that term as
\begin{equation*}
\frac{2^{\frac14}\pi}{3^{\frac12} (24n+1)^{\frac34}} 
 \int_{-1}^1\left(1-x^2\right)^{\frac34} \cot\left(\tfrac{\pi}{2}\left(\tfrac{x}{\sqrt{6}}+\tfrac12\right)\right) I_{\frac32}\left(\tfrac{2\pi}{\sqrt{3}}\sqrt{\left(1-x^2\right)\left(n+\tfrac{1}{24}\right)}\right) \ddd x.
\end{equation*} 
Using the saddle point method for the integral then yields
\begin{equation*}
\alpha_f(n) = \frac{1}{8\cdot3^{\frac 34}  n^{\frac54}}   e^{2\pi\sqrt{\frac n3}} \lp 1 +O \lp n^{-\frac{1}{2}} \rp \rp .
\end{equation*}
From this we conclude the claim.
\end{proof}

\subsection{Numerical Results}
Here we display some numerical results and compare $\a_f(n)$ for a number of cases to the result obtained by using equation \eqref{eq:f_final_result} and numerically performing the sum over $k$ from $1$ to $N$. 
\begin{table}[h]
\centering
\begin{tabular}{c | c | c | c | c | c}
  & $N=1$ & $N=2$  &   $N=3$ & $N=4$ & $N=20$ \\
  \hline
$\a_f (0) = 1$ &
 $0.536184\ldots$   &  $0.660506\ldots$   &  $0.756302\ldots$   &
$0.799454\ldots$    & $0.954218\ldots$  \\
$\a_f (7) = 63$ &
 $63.60062\ldots$   &  $63.00845\ldots$   &  $62.90648\ldots$   &
$62.90906\ldots$    & $62.96324\ldots$  \\
$\a_f (9) = 170$ &
 $170.6548\ldots$   &  $169.7915\ldots$   &  $170.0395\ldots$   &
$170.0367\ldots$    & $170.0011\ldots$  \\
$\a_f (10) = 272$ &
 $271.1167\ldots$   &  $272.1510\ldots$   &  $272.0148\ldots$   &
$271.9349\ldots$    & $272.0002\ldots$  \\
$\a_f (15) = 2191$ &
 $2192.974\ldots$   &  $2190.577\ldots$   &  $2191.006\ldots$   &
$2191.010\ldots$    & $2191.033\ldots$  \\
$\a_f (19) = 9592$ &
 $9596.754\ldots$   &  $9592.326\ldots$   &  $9592.026\ldots$   &
$9592.030\ldots$    & $9592.001\ldots$  \\
$\a_f (20) = 13602$ &
 $13596.99\ldots$   &  $13602.12\ldots$   &  $13601.79\ldots$   &
$13601.92\ldots$    & $13601.99\ldots$  \\
\hline 
\bottomrule
\end{tabular}
\caption{Numerical results for Fourier coefficients of $f(\tau) = - \frac{i}{2}  \frac{\psi(\tau)}{\eta(\tau)^2}$.}
\label{tab:a30_5}
\end{table}

\section{Proof of Theorem \ref{QuantumTheorem} and Corollary \ref{quantum}}\label{sec:Quantum}
In this section we study the family of false theta functions $F_{j,N}$ defined in \eqref{Fjp}\footnote{Note that in \cite{BM} (using $p$ instead of $N$) the function $F_{j,N}$ is defined as $\sum_{n \in \IZ} \sgn (n+\frac12)  q^{( n + \frac{j}{2N} )^2}$ and it corresponds to $F_{j,N} ( \frac{\tau}{N} )$ in our notation.} and in particular prove Theorem \ref{QuantumTheorem} and Corollary \ref{quantum}. Define the modular completion
\begin{align*}
\wh{F}_{j,N} (\tau,w) &\coloneqq \sum_{\substack{n\in\IZ\\n\equiv j\pmod{2N}}} \erf \lp - i \sqrt{\pi i (w-\t)}  \tfrac{n}{\sqrt{2N}} \rp q^{\frac{n^2}{4N}}.
\end{align*}
Observe that for $\varepsilon>0$ we have
\begin{equation}\label{eq:Fhat_limit}
%LABEL NEEDED FOR DETAILS FILE
\lim_{t\to\infty}  \wh{F}_{j,N} (\tau,\tau+it\pm\varepsilon) =\pm 
F_{j,N} (\tau).
\end{equation}
Theorem \ref{ThPh} immediately implies the following modular transformation properties.
\begin{lem}\label{cor:QuantumThm}
	The function $\wh{F}_{j,N}(\t,w)$ transforms as
	\begin{align*}
	\wh{F}_{j,N}(\tau+1,w+1) &= e^{\frac{2 \pi ij^2}{4N}} \wh{F}_{j,N}(\tau,w), \\
	\wh{F}_{j,N} \lp -\tfrac{1}{\tau},-\tfrac{1}{w} \rp 
	&= \chi_{\t,w} (-i)^{\frac{3}{2}} \tau^{\tfrac{1}{2}}  
	\sqrt{\frac{2}{N}}   \sum_{r=1}^{N-1} \sin \lp \tfrac{\pi jr}{N} \rp  
	\wh{F}_{r,N}(\tau,w) .
	\end{align*}
\end{lem}

Next we formulate a proposition that clarifies the relation between $\widehat{F}_{j,N}$ and $f_{j,N}$ in terms of Eichler-type integrals. The proof follows by a direct calculation.
\begin{prop}\label{fhat}
	For $\t, w \in \mathbb{H}$ and $\varepsilon >0$ we have
	\begin{align}
	\widehat{F}_{j,N}(\tau,w) &= \pm  F_{j,N}(\tau) - \sqrt{2N} \int_{w}^{\tau+i\infty\pm\varepsilon} \frac{f_{j,N}(\mathfrak z)}{\sqrt{i(\mathfrak z- \tau)}}\ddd\mathfrak z = \sqrt{2N} \int_{\t}^{w} \frac{f_{j,N}(\mathfrak z)}{\sqrt{i(\mathfrak z- \tau)}}\ddd\mathfrak z  ,
	\label{eq:Fhat_hol_comp}
	\\
	F_{j,N}(\tau) &= 
	-i \sqrt{2N} \int_{\t}^{i\infty} \frac{f_{j,N}(\mathfrak z)}{\sqrt{-i(\zz- \t)}}\ddd\zz.
	\label{eq:F_integral_form}
	\end{align}
\end{prop}
\begin{rem}
The identity \eqref{eq:F_integral_form} is equivalent to writing
\begin{equation*}
F_{j,N}(\tau) = 
\sqrt{2N} \int_{0}^{\infty} \frac{f_{j,N}(\t + i v)}{\sqrt{v}}   \ddd v,
\end{equation*}
which is also easy to directly deduce by a term-by-term integration using
\begin{equation*}
\sgn(n) q^{\frac{n^2}{2}} = n \int_0^\infty \frac{e^{\pi in^2(\tau+iv)}}{\sqrt{v}}dv.
\end{equation*}
This gives us another perspective on how the modular completion of false theta functions works. In the form
\begin{equation*}
F_{j,N}(\tau) = \pm \sqrt{2N} \int_{\t}^{\tau+i\infty\pm\varepsilon} \frac{f_{j,N}(\mathfrak z)}{\sqrt{i(\mathfrak z- \tau)}}\ddd\zz,
\end{equation*}
the reason for modular non-invariance is the fact that the upper-bound in the integral does not transform when $\SL_2(\IZ)$ acts on $\tau$. The completion introduced in this paper works by replacing the upper bound with an arbitrary $w \in \mathbb{H}$ that transforms in the same way as $\tau$ under $\SL_2(\IZ)$.

A related argument is used for indefinite theta functions, which yield mock modular forms. There, the convergence of indefinite theta functions is ensured by inserting combinations of factors such as $\sgn ( B (\b c,\b n) )$ where $\b n$ is a (shifted) lattice element and $\b c$ is a vector satisfying $B(\b c,\b c) < 0$ (or $=0$ in limiting cases).\footnote{These vectors need to satisfy more conditions to ensure convergence but we only focus on the aspect that their norm-squared is non-positive.} In this case, a similar identity involving the sign function is then
\begin{equation*}
\sgn(n) q^{-\frac{n^2}{2}} = n \int_0^\infty \frac{e^{- \pi i n^2 (\tau - iv)}}{\sqrt{v}} \ddd v
= -i\int_{-\t}^{i \infty} \frac{n   e^{\pi i n^2\zz}}{\sqrt{-i (\zz + \t)}} \ddd \zz.
\end{equation*}
Of course, if $\mathrm{Im}(\zz) > 0$, summing $n e^{\pi i n^2\zz }$ gives rise to weight $\frac{3}{2}$ unary theta functions. Similar to the argument for false theta functions above, the reason for non-modularity comes from the fact that the upper bound in this integral does not transform under modular transformations. Thus the resolution is similarly to insert an arbitrary $w \in \mathbb{H}$ as the upper bound and transform it like $-\tau$ under the modular group. That $w$ should transform like $-\tau$, and not like $\tau$ is what distinguishes the completion for indefinite theta functions from those for the false theta functions. Indeed it is this fact that allows one to insert $w = - \bar{\tau}$, giving the more familiar notion of mock modular forms. Another significant difference between the two cases is that because $-\tau$ lives in the lower half-plane, the integral that yields the sign functions can not be commuted with the series defining the lattice sum. That this is possible for the false case allows us to write both the false theta function and its completion in terms of Eichler-type integrals. This is in a sense why false theta functions are simpler than their indefinite counterparts.
\end{rem}
We are now ready to prove Theorem \ref{QuantumTheorem}.
\begin{proof}[Proof of Theorem \ref{QuantumTheorem}]
We start by recalling the following transformation properties of $f_{j,N}$.\footnote{ The function $\tau\mapsto f_{j,N} (\tau)$ as defined in \cite{BM} corresponds to $f_{j,N} ( \frac{\tau}{N} )$ here. Modular transformations for both $F_{j,N}$ and $f_{j,N}$ act more naturally in our notation.} It is a classic result that this is a vector-valued modular form \cite{Sh} satisfying
\begin{align*}
f_{j,N}(\tau+1) &= e^{\frac{2 \pi i j^2}{4N}}   f_{j,N}(\tau)  \andd
f_{j,N} \lp -\tfrac{1}{\tau} \rp 
=  
\sqrt{\frac{2}{N}}  (-i \tau)^{\frac{3}{2}} \sum_{r=1}^{N-1} \sin \lp \tfrac{\pi jr}{N} \rp  
f_{r,N}(\tau) .
\end{align*}
For reference we give the full multiplier system for $M = \pmat{a & b \\ c& d} \in \mathrm{SL}_2 (\IZ)$ (see, for example, \cite{CohenStromberg} for further details), where $j,r \in \{ 1, \ldots, N-1 \}$,
\begin{equation}\label{definemul}
\psi_{j,r} (M) := \begin{cases}
e^{2\pi i a b \frac{j^2}{4N}}   e^{-\frac{\pi i}{4} (1 - \sgn (d))} \d_{j,r} \qquad & \mbox{if } c=0, \\
e^{-\frac{3\pi i}{4} \sgn (c)} \sqrt{\frac{2}{N |c|}}  
\displaystyle\sum_{k=0}^{|c| - 1}  e^{\frac{\pi i}{2 N c} \lp a (2Nk + j)^2 + d r^2 \rp}   
\sin \lp \tfrac{\pi r ( 2 Nk + j)}{N |c|} \rp
\qquad & \mbox{if } c\neq 0.
\end{cases}
\end{equation}
Here $\delta_{j,r}:=1$ if $j=r$ and $0$ otherwise. Therefore we have, for $M \coloneqq \pmat{a & b \\ c & d} \in \mathrm{SL}_2(\IZ)$,
\begin{equation*}
f_{j,N} \lp \tfrac{a \tau + b}{c \tau + d}  \rp = (c \tau +d)^{\frac{3}{2}}  
\sum_{r=1}^{N-1} \psi_{j,r} (M)    f_{r,N}(\tau).
\end{equation*}
Moreover Lemma \ref{cor:QuantumThm} implies that
\begin{equation}\label{Fhat_general_transform}
\wh{F}_{j,N} \lp \tfrac{a \tau + b}{c \tau + d} ,\tfrac{a w + b}{c w + d} \rp = 
\chi_{\t,w} (M)   (c \tau +d)^{\frac{1}{2}}  
\sum_{r=1}^{N-1} \psi_{j,r} (M)    \wh{F}_{r,N}(\tau,w).
\end{equation}

Now, the $c=0$ case of Theorem \ref{QuantumTheorem} directly follows by inspection using \eqref{definemul}. So we assume from now on $c\neq 0$.
Using \eqref{eq:Fhat_hol_comp} with a $\pm$ sign to be chosen below, we write the modular transformation identity  \eqref{Fhat_general_transform} for $M = \pmat{a & b \\ c & d} \in \mathrm{SL}_2 (\IZ)$ (and $\varepsilon_1, \varepsilon_2 > 0$) as,
\begin{align*}
&\pm F_{j,N} \lp \tfrac{a \tau + b}{c \tau + d}  \rp 
- 
\sqrt{2N} \int_{\frac{aw+b}{cw+d}}^{\frac{a\t+b}{c\t+d}+i\infty\pm\varepsilon_1} \frac{f_{j,N}(\mathfrak z)}{\sqrt{i \left(\zz- \tfrac{a\t+b}{c\t+d} \right)  }}\ddd\zz   \\
&\qquad \qquad
= 
\chi_{\t,w}(M)   (c \tau +d)^{\frac{1}{2}}  
\sum_{r=1}^{N-1} \psi_{j,r} (M) 
\lp  F_{r,N}(\tau)
-
\sqrt{2N} \int_{w}^{\tau+i\infty+\varepsilon_2} \frac{f_{r,N}(\mathfrak z)}{\sqrt{i(\mathfrak z- \tau)}}\ddd\zz \rp.
\end{align*}
We then take $w \to \t + i \infty + \e_2$. Thus we get, using that $\chi_{\t,w} (M) \to \sgn (c)$,
\begin{equation*}
\pm F_{j,N} \lp \tfrac{a \tau + b}{c \tau + d}  \rp  
-  \sgn (c)
(c \tau +d)^{\frac{1}{2}} \sum_{r=1}^{N-1} \psi_{j,r}(M)   F_{r,N}(\tau)
= 
\sqrt{2N} \int_{\frac{a}{c}}^{\frac{a\t+b}{c\t+d}+i\infty\pm\varepsilon_1} \frac{f_{j,N}(\mathfrak z)}{\sqrt{i \left(\zz- \tfrac{a\t+b}{c\t+d} \right)  }}\ddd\zz   .
\end{equation*}
Now note that $\frac{a}{c} - \frac{a\t+b}{c \t+d} = \frac{1}{c^2} \frac{1}{\t + \frac{d}{c}}$.
So supposing that $\t_1 \neq - \frac{d}{c}$, if we let $\varepsilon_1 \coloneqq \frac{1}{c^2} \frac{|\t_1 + \frac{d}{c} |}{| \t + \frac{d}{c} |^2}$ and $\pm = \sgn (\t_1 + \frac{d}{c})$, then we can choose the integration path to be a vertical path from $\frac{a}{c}$ to $i \infty$. Making that choice we get
\begin{equation*}
F_{j,N} \lp \tfrac{a \tau + b}{c \tau + d}  \rp  
-  \sgn (c \t_1 + d)
(c \tau +d)^{\frac{1}{2}} \sum_{r=1}^{N-1} \psi_{j,r}(M)   F_{r,N}(\tau)
=   \sgn \lp \t_1 + \tfrac{d}{c} \rp
\sqrt{2N} \int\displaylimits_{\frac{a}{c}}^{\frac{a}{c}+i\infty} \frac{f_{j,N}(\mathfrak z)}{\sqrt{i \left(\zz- \frac{a\t+b}{c\t+d} \right)  }}\ddd\zz   .
\end{equation*}
Since on this vertical path we have $\sqrt{i (\zz- \frac{a\t+b}{c\t+d} ) } =\sgn(\tau_1+\frac{d}{c})   i   \sqrt{-i (\zz- \frac{a\t+b}{c\t+d}) }$ we get
\begin{equation*}
F_{j,N} \lp \tfrac{a \tau + b}{c \tau + d}  \rp  
-  \sgn (c \t_1 + d)   
(c \tau +d)^{\frac{1}{2}} \sum_{r=1}^{N-1} \psi_{j,r}(M)     F_{r,N}(\tau)
=   -i
\sqrt{2N} \int_{\frac{a}{c}}^{i\infty} \frac{f_{j,N}(\mathfrak z)}{\sqrt{-i \left(\zz- \frac{a\t+b}{c\t+d} \right)  }}\ddd\zz   ,
\end{equation*}
for any integration path that avoids the branch-cut.\footnote{Note that the integrand is exponentially decaying as $\mathrm{Im} (\zz) \to \infty$, so after rotating the branch-cut we can freely shift the upper-bound in the integral.} Sending $M \mapsto M^{-1}$ we obtain
\begin{align*}
F_{j,N} \lp \tfrac{d \tau - b}{-c \tau + a}  \rp  
-  \sgn (-c \t_1 + a)  
(-c \tau +a)^{\frac{1}{2}} &\sum_{r=1}^{N-1} \psi_{j,r}\lp M^{-1}\rp  F_{r,N}(\tau)
\notag \\
& \quad
  =   -i
\sqrt{2N} \int_{-\frac{d}{c}}^{i\infty} \frac{f_{j,N}(\mathfrak z)}{\sqrt{-i \left(\zz- \frac{d\t-b}{-c\t+a} \right)  }}\ddd\zz   .
\label{F_transformation}
\end{align*}
Replacing $\t \mapsto M\tau$ yields the claim.
\end{proof}

\begin{proof}[Proof of Corollary \ref{quantum}]
	The claim follows from taking vertical limits of $\tau$ to rationals in the statement of Theorem \ref{QuantumTheorem} and noting that the integral giving the obstruction to modularity extends to $\mathbb{R}$ as a function of $\tau$ and is real-analytic in $\mathbb{R} \setminus \{-\tfrac{d}{c} \}$.
\end{proof}

\end{document}